\documentclass[11pt]{ip-journal}

\usepackage[mathscr]{euscript}
\DeclareFontFamily{OT1}{rsfs}{}
\DeclareFontShape{OT1}{rsfs}{n}{it}{<-> rsfs10}{}
\DeclareMathAlphabet{\curly}{OT1}{rsfs}{n}{it}

\makeatletter
\newcommand{\eqnum}{\refstepcounter{equation}\textup{\tagform@{\theequation}}}
\makeatother

\DeclareRobustCommand{\SkipTocEntry}[3]{}
\makeatletter
\newcommand\@dotsep{4.5}
\def\@tocline#1#2#3#4#5#6#7{\relax
  \ifnum #1>\c@tocdepth % then omit
  \else
    \par \addpenalty\@secpenalty\addvspace{#2}%
    \begingroup \hyphenpenalty\@M
    \@ifempty{#4}{%
      \@tempdima\csname r@tocindent\number#1\endcsname\relax
    }{%
      \@tempdima#4\relax
    }%
    \parindent\z@ \leftskip#3\relax \advance\leftskip\@tempdima\relax
    \rightskip\@pnumwidth plus1em \parfillskip-\@pnumwidth
    #5\leavevmode #6\relax
    \leaders\hbox{$\m@th
      \mkern \@dotsep mu\hbox{.}\mkern \@dotsep mu$}\hfill
    \hbox to\@pnumwidth{\@tocpagenum{#7}}\par
    \nobreak
    \endgroup
  \fi}
\makeatother

\newcommand\beq[1]{\begin{equation}\label{#1}}
\newcommand\eeq{\end{equation}}
\newcommand\beqa{\begin{eqnarray*}}
\newcommand\eeqa{\end{eqnarray*}}

\title[GV invariants and wall-crossing]{Gopakumar-Vafa invariants and wall-crossing}
\date{}
\author{Yukinobu Toda}

\usepackage{amscd}
\usepackage{amsmath}
\usepackage{amssymb}
\usepackage{amsthm}
\usepackage{float}
\usepackage[dvips]{graphicx}
\usepackage{xypic}

\usepackage{array}
\usepackage{amscd}
\usepackage[all]{xy}

\DeclareFontFamily{U}{rsfs}{%
\skewchar\font127}
\DeclareFontShape{U}{rsfs}{m}{n}{%
<-6>rsfs5<6-8.5>rsfs7<8.5->rsfs10}{}
\DeclareSymbolFont{rsfs}{U}{rsfs}{m}{n}
\DeclareSymbolFontAlphabet
{\mathrsfs}{rsfs}
\DeclareRobustCommand*\rsfs{%
\@fontswitch\relax\mathrsfs}

\theoremstyle{plain}
\newtheorem{thm}{Theorem}[section]
\newtheorem{prop}[thm]{Proposition}
\newtheorem{lem}[thm]{Lemma}

\newtheorem{defi}[thm]{Definition}
\newtheorem{rmk}[thm]{Remark}
\newtheorem{cor}[thm]{Corollary}

\newtheorem{prop-defi}[thm]{Proposition-Definition}
\newtheorem{thm-defi}[thm]{Theorem-Definition}
\newtheorem{lem-defi}[thm]{Lemma-Definition}

\newtheorem{conj}[thm]{Conjecture}
\newtheorem{exam}[thm]{Example}

\newcommand{\sslash}{/\!\!/}

\newcommand{\aA}{\mathcal{A}}

\newcommand{\cC}{\mathcal{C}}

\newcommand{\eE}{\mathcal{E}}
\newcommand{\fF}{\mathcal{F}}

\newcommand{\hH}{\mathcal{H}}

\newcommand{\lL}{\mathcal{L}}
\newcommand{\mM}{\mathcal{M}}

\newcommand{\oO}{\mathcal{O}}

\newcommand{\sS}{\mathcal{S}}

\newcommand{\uU}{\mathcal{U}}

\newcommand{\Supp}{\mathop{\rm Supp}\nolimits}
\newcommand{\Hom}{\mathop{\rm Hom}\nolimits}

\newcommand{\dR}{\mathbf{R}}
\newcommand{\dL}{\mathbf{L}}

\newcommand{\Pic}{\mathop{\rm Pic}\nolimits}

\newcommand{\Chow}{\mathop{\rm Chow}\nolimits}

\newcommand{\id}{\textrm{id}}

\newcommand{\ch}{\mathop{\rm ch}\nolimits}

\newcommand{\Ext}{\mathop{\rm Ext}\nolimits}
\newcommand{\Spec}{\mathop{\rm Spec}\nolimits}

\newcommand{\Coh}{\mathop{\rm Coh}\nolimits}

\newcommand{\cneq}{\mathrel{\raise.095ex\hbox{:}\mkern-4.2mu=}}
\newcommand{\eqcn}{\mathrel{=\mkern-4.5mu\raise.095ex\hbox{:}}}

\newcommand{\codim}{\mathop{\rm codim}\nolimits}
\newcommand{\Ex}{\mathop{\rm Ex}\nolimits}

\newcommand{\Stab}{\mathop{\rm Stab}\nolimits}

\newcommand{\oPPer}{\mathop{\rm ^{0}Per}\nolimits}

\newcommand{\iPPer}{\mathop{\rm ^{-1}Per}\nolimits}

\newcommand{\ptau}{\mathop{^{{p}}\tau}\nolimits}
\newcommand{\otau}{\mathop{^{{0}}\tau}\nolimits}
\newcommand{\itau}{\mathop{^{{-1}}\tau}\nolimits}

\newcommand{\pZ}{\mathop{^{{p}}Z}\nolimits}

\newcommand{\pH}{\mathop{^{{p}}\mathcal{H}}\nolimits}

\newcommand{\Perv}{\mathop{\rm Perv}\nolimits}

\newcommand{\IC}{\mathop{\rm IC}\nolimits}

\newcommand{\pPPer}{\mathop{\rm ^{\mathit{p}}Per}\nolimits}

\newcommand{\Sym}{\mathop{\rm Sym}\nolimits}

\newcommand{\Imm}{\mathop{\rm Im}\nolimits}

 \newcommand{\RHom}{\mathop{\dR\mathrm{Hom}}\nolimits}

\newcommand{\Ree}{\mathop{\rm Re}\nolimits}

\newcommand{\GL}{\mathop{\rm GL}\nolimits}

\newcommand{\tr}{\mathop{\rm tr}\nolimits}

\newcommand{\length}{\mathop{\rm length}\nolimits}

\renewcommand{\labelenumi}{(\roman{enumi})}

\newcommand{\lkakko}{[\![}
\newcommand{\rkakko}{]\!]}

\makeatletter
 \renewcommand{\theequation}{%
   \thesection.\arabic{equation}}
  \@addtoreset{equation}{section}
\makeatother

\begin{document}

\begin{abstract}
In this paper, we 
generalize a mathematical definition of 
Gopakumar-Vafa (GV) invariants on Calabi-Yau 3-folds
introduced by Maulik and the author, using 
an analogue of BPS sheaves 
introduced by Davison-Meinhardt on
the coarse
 moduli spaces of one dimensional 
twisted semistable 
sheaves with arbitrary 
holomorphic Euler characteristics. 
We show that our generalized GV invariants 
are
independent of twisted stability conditions, 
and conjecture that they are 
also independent of holomorphic
Euler characteristics, so that 
they define the same 
GV invariants. As an application, we will 
show the flop transformation formula of GV invariants. 
\end{abstract}

\maketitle

\section{Introduction}
\subsection{Background}
Let $X$ be a smooth projective Calabi-Yau 3-fold over $\mathbb{C}$. For
$g \in \mathbb{Z}_{\ge 0}$ and 
$\beta \in H_2(X, \mathbb{Z})$, 
Gopakumar-Vafa~\cite{GV} conjectured the existence of integer valued 
invariants (called \textit{Gopakumar-Vafa (GV) invariants})
\begin{align}\label{intro:GV}
n_{g, \beta} \in \mathbb{Z}, \ g \in \mathbb{Z}_{\ge 0}, \ \beta \in H_2(X, \mathbb{Z})
\end{align}
which determine Gromov-Witten invariants~\cite{BGW}
 and Pandharipande-Thomas invariants~\cite{PT} by 
taking their generating series (see~\cite[Section~3.3]{MT} for precise 
conjectures). 
The original approach
 of Gopakumar-Vafa~\cite{GV}
toward defining (\ref{intro:GV})
was to use the $sl_2 \times sl_2$-action on the cohomology of 
the moduli space of 
D2-branes, which
may be mathematically interpreted as the moduli space of 
one dimensional (semi)stable sheaves on $X$. 
In~\cite{MT}, Maulik and the author proposed 
a mathematical definition of the invariants (\ref{intro:GV})
along with the idea of Gopakumar-Vafa~\cite{GV}, 
based on earlier works 
by Hosono-Saito-Takahashi~\cite{HST}, 
Katz~\cite{Katz} (the $g=0$ case) 
and Kiem-Li~\cite{KL}. 
As we will review shortly, the 
 key ingredients of the definition in~\cite{MT} are
the perverse sheaf of vanishing cycles 
on the moduli space of one dimensional stable sheaves, 
and the character formula of $sl_2 \times sl_2$-action
on its cohomology.   

The relevant moduli space in the above works~\cite{HST, Katz, KL, MT}
is the 
moduli space $M(\beta, 1)$
of one dimensional Gieseker-stable 
sheaves $E$ on $X$ satisfying $[l(E)]=\beta$
and $\chi(E)=1$, where $l(E)$ is the fundamental one cycle of $E$. 
In this paper, we address the question 
whether we can 
also define the invariants (\ref{intro:GV}) using 
some variants of the moduli space $M(\beta, 1)$, i.e. 
different stability conditions, holomorphic Euler characteristics.  
For $m\in \mathbb{Z}$ and 
an element of the 
complexified ample cone
\begin{align}\label{intro:sigma}
\sigma=B+i\omega \in A(X)_{\mathbb{C}}
\end{align}
let $M_{\sigma}(\beta, m)$
be the coarse moduli space 
of $B$-twisted $\omega$-semistable
one dimensional sheaves
$E$ on $X$ satisfying $[l(E)]=\beta$
and $\chi(E)=m$.
The moduli space $M_{\sigma}(\beta, m)$
is related to 
the previous one by 
$M(\beta, 1)=M_{\sigma=i\omega}(\beta, 1)$. 
The purpose of this paper is to 
extend the construction of 
GV invariants in~\cite{MT}
using
the 
moduli space
$M_{\sigma}(\beta, m)$,
and 
study the 
independence of the resulting invariants of 
additional data $(\sigma, m)$.

\subsection{Generalized GV invariants}
In~\cite{MT}, by modifying the earlier work
of Kiem-Li~\cite{KL}, 
we defined the invariant (\ref{intro:GV})
using a perverse sheaf of vanishing cycles on 
$M(\beta, 1)$, and the perverse cohomologies 
of its push-forward to 
the Chow variety $\Chow_X(\beta)$. Here
the latter moduli space 
parametrizes effective one cycles on $X$
with homology class $\beta$. 
The basic fact behind this idea is that 
$M(\beta, 1)$ is the truncation of a 
derived scheme with a $(-1)$-shifted symplectic 
structure~\cite{PTVV}, so by the derived 
Darboux theorem~\cite{BBBJ}
it has a $d$-critical structure 
introduced by Joyce~\cite{JoyceD}. 
In particular 
it is locally written as a critical locus of 
some functions. The locally defined 
vanishing cycle sheaves can be glued if we choose an 
orientation data, which is a square root line bundle 
of the virtual canonical line bundle of $M(\beta, 1)$. 

We will apply the similar idea for the 
moduli space $M_{\sigma}(v)$, where 
$\sigma$ is an element (\ref{intro:sigma})
and $v=(\beta, m)$. 
A difference from the moduli space $M(\beta, 1)$ is 
that $M_{\sigma}(v)$ is not a fine moduli space in general, 
so it is not locally written as a critical locus. 
Nevertheless, we will define the perverse sheaf\footnote{The definition of
the perverse sheaf (\ref{intro:phiM}) was suggested to 
the author by Ben Davison.}
\begin{align}\label{intro:phiM}
\phi_{M_{\sigma}(v)} \in \Perv(M_{\sigma}(v))
\end{align}
as an analogue of BPS sheaves
introduced by Davison-Meinhardt~\cite{DaMe} on 
the coarse moduli spaces of semistable quiver 
representations with super-potentials. 
The perverse sheaf (\ref{intro:phiM}) is, roughly speaking, 
constructed 
as follows. 
Let $\mM_{\sigma}(v)$ 
be the moduli stack of $\sigma$-semistable sheaves on $X$ 
with Chern character $v$. 
Then the stack $\mM_{\sigma}(v)$ admits a $d$-critical 
structure~\cite{BBBJ}, so 
given an orientation data 
we can construct 
a perverse sheaf of vanishing cycles on the stack $\mM_{\sigma}(v)$. 
Then we push-forward it to the coarse moduli space $M_{\sigma}(v)$, 
and its 
first perverse cohomology defines the
perverse sheaf (\ref{intro:phiM}).  
It will turn out that (\ref{intro:phiM}) is an analogue of 
BPS sheaves (see Subsection~\ref{subsec:BPSsheaf}). 

Let $\pi_M$ be the Hilbert-Chow map 
\begin{align*}
\pi_M \colon M_{\sigma}(v) \to \Chow_X(\beta)
\end{align*}
sending $E$ to its fundamental one cycle. 
Then as a generalization of the construction in~\cite{MT}, 
we introduce the following definition:
\begin{defi}\emph{(Definition~\ref{def:def:phiM})}
For $\gamma \in \Chow_X(\beta)$, we define the invariant
\begin{align}\label{intro:def:Phi}
\Phi_{\sigma}(\gamma, m) \cneq 
\sum_{i\in \mathbb{Z}}
\chi(\pH^i(\dR \pi_{M\ast} \phi_{M_{\sigma}(v)})|_{\gamma})
y^i \in \mathbb{Z}[y^{\pm 1}]. 
\end{align}
\end{defi}
We note that the GV invariant
$n_{g, \beta}$ defined in~\cite{MT}
is recovered from (\ref{intro:def:Phi}) by the  
character formula of $sl_2 \times sl_2$-action
\begin{align*}
\int_{\gamma \in \Chow_X(\beta)}
\Phi_{\sigma=i\omega}(\gamma, 1)  \ de=
\sum_{g\ge 0}n_{g, \beta}(y^{1/2}+y^{-1/2})^{2g}. 
\end{align*}
The following is the main conjecture we address in this paper. 
\begin{conj}\label{intro:conj}
The invariant 
$\Phi_{\sigma}(\gamma, m)$ is independent of 
$\sigma$ and $m$. 
\end{conj}
If the above conjecture is true, then one 
can define the GV invariants (\ref{intro:GV}) from 
the moduli space $M_{\sigma}(v)$
for an arbitrary element $\sigma$ in (\ref{intro:sigma})
and $v=(\beta, m)$.
Namely if we define $n_{g, \beta, m}(\sigma)$
by the identity
\begin{align*}
\int_{\gamma \in \Chow_X(\beta)}
\Phi_{\sigma}(\gamma, m)  \ de=
\sum_{g\ge 0}n_{g, \beta, m}(\sigma)(y^{1/2}+y^{-1/2})^{2g}
\end{align*}
then $n_{g, \beta, m}(\sigma)$ is independent of $\sigma$ and $m$, 
so we have $n_{g, \beta, m}(\sigma)=n_{g, \beta}$.
When $(\gamma, m)$ is coprime, then Conjecture~\ref{intro:conj}
was also 
proposed in~\cite[Conjecture~3.21]{MT}, and Conjecture~\ref{intro:conj} is its generalization 
to the non-coprime $(\gamma, m)$ case. 
A similar result and conjecture were discussed in~\cite{JS, Todpara}
for generalized Donaldson-Thomas invariants~\cite{JS, K-S} 
counting one dimensional semistable sheaves, 
which in turn implies Pandharipande-Thomas's 
strong rationality conjecture~\cite{PT}
 (see~\cite{Tsurvey} for details).

Here we mention 
about a technical subtlety 
in defining the invariant (\ref{intro:def:Phi}). 
By their constructions, the perverse sheaf (\ref{intro:phiM}) 
and the invariant (\ref{intro:def:Phi})
depend on a choice of an orientation data of $\mM_{\sigma}(v)$, 
so we have to 
specify its choice. 
Similarly to~\cite{MT}, we 
impose the condition on an orientation data of $\mM_{\sigma}(v)$ 
so that it is trivial locally on the Chow variety. 
More precisely, 
we assume that the virtual canonical 
line bundle of the stack $\mM_{\sigma}(v)$ 
is trivial on the preimage of 
some open neighborhood $\gamma \in U \subset \Chow_X(\beta)$
under the Hilbert-Chow map
\begin{align*}
\pi_{\mM} \colon 
\mM_{\sigma}(v) \to \Chow_X(\beta).
\end{align*}
In this case, we
say that $\mM_{\sigma}(v)$ is \textit{CY at $\gamma$}, and 
conjecture that this is always the case. 
 Then 
we can take an
orientation data
of $\pi_{\mM}^{-1}(U)$ 
which is trivial as a line bundle.
Such an orientation data is called 
a \textit{CY orientation data}. 
Then by a local argument in an open neighborhood 
of $\gamma$, we can define the invariant (\ref{intro:def:Phi})
using a CY orientation data. 
The resulting invariant (\ref{intro:def:Phi}) is 
shown to be independent of 
a
CY orientation data (see Lemma~\ref{lem:inde}).

\subsection{Results}
We will study Conjecture~\ref{intro:conj} 
under a CY condition on the bigger stack $\mM_{X}(\beta)$
of pure one dimensional sheaves $E$ with $[l(E)]=\beta$. 
Our main result is the
 independence of stability conditions in Conjecture~\ref{intro:conj}: 
\begin{thm}\emph{(Theorem~\ref{thm:inde})}\label{intro:thm1}
Let $X$ be a smooth projective CY 3-fold. 
For an effective one cycle $\gamma$ on $X$ with 
homology class $\beta$, suppose that the stack 
$\mM_X(\beta)$ is CY at $\gamma$. 
Then the invariant $\Phi_{\sigma}(\gamma, m)$
is independent of $\sigma$. 
\end{thm}
By the above theorem, if $\mM_X(\beta)$ is CY at $\gamma$
we can write
\begin{align*}
\Phi_{X}(\gamma, m) \cneq \Phi_{\sigma}(\gamma, m). 
\end{align*}
The result of Theorem~\ref{intro:thm1}
implies a trivial wall-crossing of  
the invariants $\Phi_{\sigma}(\gamma, m)$. 
Namely there is 
a locally finite number of codimension one submanifolds in $A(X)_{\mathbb{C}}$ called 
\textit{walls} such that $M_{\sigma}(v)$ is constant 
if $\sigma$ lies
on a connected component of complement of walls (called a \textit{chamber})
but may change if $\sigma$ crosses a wall. 
The result of Theorem~\ref{intro:thm1} implies that, 
although the moduli space $M_{\sigma}(v)$ may change 
by wall-crossing, the associated invariants 
$\Phi_{\sigma}(\gamma, m)$ are not changed, i.e. 
the wall-crossing formula is trivial. 

The main idea of 
the proof of Theorem~\ref{intro:thm1} is to reduce 
to the case of representations of quivers with formal but convergent super-potentials. 
In this reduction step, we use the  
result of the companion paper~\cite{Todstack}, where we prove that the 
moduli stack of semistable sheaves $\mM_{\sigma}(v)$ is
described analytic locally on $M_{\sigma}(v)$ as
the moduli stack of representations of 
the Ext-quiver with a convergent super-potential. 
Here the super-potential is 
defined from the minimal $A_{\infty}$-structure
of the derived category of coherent sheaves on $X$. 
After the above reduction, we use the
results and arguments of Davison-Meinhardt~\cite{DaMe}, 
where a similar 
wall-crossing phenomena was investigated
for representations of quivers with super-potentials.  

The first application of Theorem~\ref{intro:thm1} is to show
 the independence of $m$ of the 
invariant $\Phi_{\sigma}(\gamma, m)$, when 
$\gamma$ is a primitive one cycle, i.e. 
$\gamma$ is written as 
$\sum_{1\le i \le k} a_i[C_i]$ for irreducible curves $C_i$, 
$a_i \in \mathbb{Z}_{\ge 1}$
with 
$\mathrm{g.c.d.}(a_1, \ldots, a_k)=1$. 
\begin{thm}\emph{(Theorem~\ref{thm:indeE})}\label{intro:thm2}
Under the situation of Theorem~\ref{intro:thm1}, suppose that 
$\gamma$ is a primitive one cycle. 
Then $\Phi_X(\gamma, m)$ is independent of $m$. 
\end{thm}
The next application of Theorem~\ref{intro:thm1}
is to show the flop invariance of the invariant $\Phi_X(\gamma, m)$. 
Let 
\begin{align}\label{intro:flop:dia}
\phi \colon X \stackrel{f}{\to} Y 
\stackrel{f^{\dag}}{\leftarrow} X^{\dag}
\end{align}
be a flop between smooth projective CY 3-folds. 
In this situation, we 
have the following result: 
\begin{thm}\emph{(Theorem~\ref{thm:flop})}\label{thm:intro:flop}
Let 
$\gamma$ be an effective one cycle on $X$
with homology class $\beta$
such that $f_{\ast}\gamma \neq 0$. 
Suppose 
that the stacks
$\mM_X(\beta)$, $\mM_{X^{\dag}}(\phi_{\ast}\beta)$
are CY at 
$\gamma$, $\phi_{\ast}\gamma$ respectively. 
Then we have the identity 
\begin{align*}
\Phi_X(\gamma, m)=\Phi_{X^{\dag}}(\phi_{\ast}\gamma, m).
\end{align*}
\end{thm}
If the assumption of Theorem~\ref{thm:intro:flop}
holds for $m=1$ and any $\gamma \in \Chow_X(\beta)$, we in particular 
obtain 
\begin{align*}
n_{g, \beta}=n_{g, \phi_{\ast}\beta}
\end{align*}
for the curve class $\beta$ with
 $f_{\ast}\beta \neq 0$. 
The above 
flop invariance of GV invariants 
was proved in~\cite{MT}
when the curve class $\beta$ is irreducible. 
The result of Theorem~\ref{thm:intro:flop}
gives a complete answer to the flop invariance of 
GV invariants for any curve class, assuming the CY properties of the 
relevant moduli stacks of one dimensional sheaves. 

So far the 
results in Theorem~\ref{intro:thm1}, Theorem~\ref{intro:thm2}
and Theorem~\ref{thm:intro:flop}
are conditional to the
conjectural 
CY property of the stack 
$\mM_X(\beta)$. 
We will show that the above CY property holds
for the non-compact CY 3-fold
\begin{align*}
X=\mathrm{Tot}_S(K_S)
\end{align*}
where $S$ is a smooth projective surface. 
Although $X$ is non-compact in this case, we can 
similarly define the invariant $\Phi_{\sigma}(\gamma, m)$
and can ask its independence of $(\sigma, m)$
as in Conjecture~\ref{intro:conj}. 
Then the analogy of the results in Theorem~\ref{intro:thm1}, 
Theorem~\ref{intro:thm2} and 
Theorem~\ref{thm:intro:flop}
hold without assuming the CY properties: 
\begin{thm}\emph{(Theorem~\ref{thm:CYsurface}, Theorem~\ref{thm:inde:loc}, Theorem~\ref{thm:floploc})}\label{intro:thm:loc}
Let $S$ be a smooth projective surface and 
$X=\mathrm{Tot}_S(K_S)$ the non-compact CY 3-fold. 
Then we have the following: 
\begin{enumerate}
\item For any effective compactly supported one cycle $\gamma$ on $X$
with homology class $\beta$, the stack $\mM_X(\beta)$ is CY at $\gamma$. 
In particular
for any element 
$\sigma=B+i\omega \in A(S)_{\mathbb{C}}$ and 
$m \in \mathbb{Z}$, the invariant 
$\Phi_{\sigma}(\gamma, m) \in \mathbb{Z}[y^{\pm 1}]$
is defined as in (\ref{intro:def:Phi}) with the CY orientation data. 
\item The invariant 
$\Phi_{\sigma}(\gamma, m)$ is independent of 
$\sigma$. So we can write it 
as $\Phi_X(\gamma, m)$. 

\item $\Phi_X(\gamma, m)$ is also independent of $m$ 
if $\gamma$ is a 
primitive one cycle. 

\item Suppose that $\gamma$ is supported on the zero section 
$S \subset X$, and let $h \colon S^{\dag} \to S$
be a blow-up at a point. 
Then for $X^{\dag}=\mathrm{Tot}_{S^{\dag}}(K_{S^{\dag}})$ 
and the one cycle $h^{\ast}\gamma$ on $X^{\dag}$ supported on
the zero section $S^{\dag} \subset X^{\dag}$, we have 
$\Phi_X(\gamma, m)=\Phi_{X^{\dag}}(h^{\ast}\gamma, m)$. 
\end{enumerate} 
\end{thm}

\subsection{Plan of the paper}
The organization of this paper is as follows. 
In Section~\ref{sec:genGV}, we introduce the 
invariant $\Phi_{\sigma}(\gamma, m)$ and propose the 
conjecture that it is independent of $\sigma$ and $m$. 
In Section~\ref{sec:exam}, we compute the 
invariant $\Phi_{\sigma}(\gamma, m)$ in some examples. 
In Section~\ref{sec:wcf}, 
we discuss a wall-crossing formula of 
perverse sheaves of vanishing cycles of 
representations of quivers with convergent super-potentials. 
In Section~\ref{sec:wcf:gv}, we prove Theorem~\ref{intro:thm1} and 
Theorem~\ref{intro:thm2}.  
In Section~\ref{sec:flop}, 
we prove Theorem~\ref{thm:intro:flop}. 
In Section~\ref{sec:local}, 
we prove Theorem~\ref{intro:thm:loc}.

\subsection{Acknowledgements}
The author is grateful to Ben Davison and Davesh Maulik
for many useful discussions. In particular, the 
definition of 
the perverse sheaf (\ref{intro:phiM}) 
was suggested by Ben Davison through the discussion with him. 
The author also thanks to anonymous referees for several 
suggestions and comments. 
The author is supported by World Premier International Research Center
Initiative (WPI initiative), MEXT, Japan, and Grant-in Aid for Scientific
Research grant (No. 26287002) from MEXT, Japan.

\subsection{Notation and convention}
In this paper, all schemes 
and stacks
are defined over $\mathbb{C}$. 
For a scheme or stack $M$, we
will only consider constructible sheaves on it 
with $\mathbb{Q}$-coefficients. 
We denote by $\Perv(M)$ the category 
of perverse sheaves on $M$, which is the heart of 
a t-structure on the derived category of 
constructible sheaves on $M$ (see~\cite{BBD, MR2480756}). 
Let $\iota \colon M^{\rm{red}} \hookrightarrow M$ be the reduced part
of $M$. 
Since $\iota$ is a homeomorphism, 
we always identity $\Perv(M)$ with $\Perv(M^{\rm{red}})$
in a natural way. 

For a bounded 
complex $E$ of constructible sheaves on $M$, 
we denote by $\pH^i(E)$ the $i$-th cohomology with respect 
to the perverse t-structure, and $\chi(E)$ is the 
the Euler characteristic of $\dR \Gamma(M, E)$. 
For a constructible function $\nu$ on a scheme $M$, 
the weighted Euler characteristic is denoted by 
\begin{align*}
\int_M \nu \ de \cneq \sum_{m \in \mathbb{Z}}
m \cdot e(\nu^{-1}(m)). 
\end{align*}
Here $e(-)$ is the topological Euler characteristic.

\section{Generalized GV invariants}\label{sec:genGV}
In this section, we recall some necessary background on 
moduli spaces of semistable sheaves~\cite{MR1450870}, 
and 
Joyce's $d$-critical schemes
and stacks~\cite{JoyceD}. We then introduce the invariant 
$\Phi_{\sigma}(\gamma, m) \in \mathbb{Z}[y^{\pm 1}]$, using 
an analogy of BPS sheaves~\cite{DaMe}. 
\subsection{Twisted semistable sheaves}\label{subsec:twist2}
Let $X$ be a smooth projective Calabi-Yau 3-fold over $\mathbb{C}$, i.e. 
$\dim X=3$ and $K_X=0$. 
We denote by 
\begin{align*}
\Coh_{\le 1}(X) \subset \Coh(X)
\end{align*} 
the abelian subcategory of 
coherent sheaves $E$ on $X$ whose supports have dimensions less than or equal to one. 
Let $A(X)_{\mathbb{C}}$ be the complexified ample cone of $X$
defined by
\begin{align*}
A(X)_{\mathbb{C}} \cneq \{B+i\omega \in \mathrm{NS}(X)_{\mathbb{C}} : \omega \mbox{ is ample }\}. 
\end{align*}
For 
an object $E \in \Coh_{\le 1}(X)$ and 
an element
\begin{align}\label{sigma:B}
B+i\omega \in A(X)_{\mathbb{C}}
\end{align}
the \textit{$B$-twisted $\omega$-slope} $\mu_{B, \omega}(E) \in \mathbb{R} \cup \{\infty\}$ 
is defined by 
\begin{align*}
\mu_{B, \omega}(E) \cneq \frac{\ch_3^B(E)}{\omega \cdot \ch_2^B(E)}
=\frac{\chi(E)-B \cdot l(E)}{\omega \cdot l(E)}. 
\end{align*}
Here $\ch^B(-) \cneq e^{-B}\ch(-)$ is the $B$-twisted Chern character
and $\mu_{B, \omega}(E)=\infty$ if $\omega \cdot \ch_2^B(E)=0$. 
Also $l(E)$ is the fundamental one cycle of $E$, defined by 
\begin{align*}
l(E) \cneq \sum_{\eta \in X, \dim \overline{\{\eta\}}=1}
 \length(E_{\eta}) \cdot \overline{\{\eta\}}. 
\end{align*}
\begin{defi}
An object $E \in \Coh_{\le 1}(X)$ is called 
$(B, \omega)$-(semi)stable if for any 
subsheaf $0 \neq F \subsetneq E$, we have 
$\mu_{B, \omega}(F)<(\le) \mu_{B, \omega}(E)$. 
\end{defi}
The above stability condition can be interpreted 
in terms of Bridgeland stability conditions~\cite{Brs1} as follows. 
Let 
\begin{align*}
N_1(X) \subset H_2(X, \mathbb{Z})
\end{align*}
 be the group of numerical classes of 
algebraic one cycles on $X$
and set 
\begin{align*}
\Gamma_X \cneq N_1(X) \oplus \mathbb{Z}.
\end{align*}
The Chern character 
of an object in $D^b(\Coh_{\le 1}(X))$ takes
its value in $\Gamma_X$, and given by
\begin{align}\label{ch:Gamma}
\ch(E)=(\ch_2(E), \ch_3(E)) =([l(E)], \chi(E)). 
\end{align}
By definition, a \textit{Bridgeland stability condition} on $D^b(\Coh_{\le 1}(X))$
w.r.t. the Chern character map (\ref{ch:Gamma})
consists of data
\begin{align}\label{def:stab}
\sigma=(Z, \aA), \ Z \colon \Gamma_X \to \mathbb{C}, \ 
\aA \subset D^b(\Coh_{\le 1}(X))
\end{align}
where $Z$ is a group homomorphism, 
$\aA$ is the heart of a bounded t-structure
satisfying some axioms (see~\cite{Brs1, K-S} for details). 
It determines the set of \textit{$\sigma$-(semi)stable objects}: 
$E \in D^b(\Coh_{\le 1}(X))$ is $\sigma$-(semi)stable if 
$E[k] \in \aA$ for some $k\in \mathbb{Z}$, and for any 
non-zero subobject $0\neq F \subsetneq E[k]$ in $\aA$, we have the 
inequality in $(0, \pi]$: 
\begin{align*}
\arg Z(\ch(F))<(\le) \arg Z(\ch(E[k])).
\end{align*}

The set of 
Bridgeland stability conditions (\ref{def:stab})
forms a complex manifold, which 
we denote by $\Stab_{\le 1}(X)$. 
The forgetting map $(Z, \aA) \mapsto Z$ gives a local 
homeomorphism
 \begin{align*}
\Stab_{\le 1}(X) \to (\Gamma_X)_{\mathbb{C}}^{\vee}. 
\end{align*}
For a given element (\ref{sigma:B}), let 
$Z_{B, \omega}$ be the group homomorphism 
$\Gamma_X \to \mathbb{C}$ defined by 
\begin{align}\label{ZBw}
Z_{B, \omega}(\beta, m) \cneq -m+(B+i\omega)\beta. 
\end{align}
Then the pair
\begin{align}\label{sigma:Bw}
\sigma_{B, \omega} \cneq (Z_{B, \omega}, \Coh_{\le 1}(X))
\end{align}
determines a point in $\Stab_{\le 1}(X)$. 

It is obvious that 
an object in $\Coh_{\le 1}(X)$ is 
$(B, \omega)$-(semi)stable iff it is 
Bridgeland $\sigma_{B, \omega}$-(semi)stable. 
We also call $(B, \omega)$-(semi)stable sheaves as 
\textit{$\sigma_{B, \omega}$-(semi)stable objects}. 
Moreover the map
\begin{align*}
A(X)_{\mathbb{C}} \to \Stab_{\le 1}(X),  \ 
(B, \omega) \mapsto \sigma_{B, \omega}
\end{align*}
is a continuous injective 
map, whose image is denoted by 
\begin{align*}
U(X) \subset \Stab_{\le 1}(X).
\end{align*}  
We sometimes write $\sigma=\sigma_{B, \omega} \in U(X)$
as $\sigma=B+i\omega$. 
\subsection{Moduli stacks of twisted semistable sheaves}\label{subsec:twist}
For $\beta \in N_1(X)$, let 
$\mM_X(\beta)$ be the 2-functor
\begin{align}\label{moduli:2funct}
\mM_X(\beta) \colon 
Sch/\mathbb{C} \to Groupoid
\end{align}
sending a $\mathbb{C}$-scheme $S$ to the 
groupoid of $S$-flat sheaves 
$\eE \in \Coh(X \times S)$ such that 
for each closed point $s \in S$, 
the sheaf $\eE_s\cneq \eE|_{X \times \{s\}}$ 
is an object in $\Coh_{\le 1}(X)$ 
satisfying $[l(\eE_s)]=\beta$. 
It is well-known that the 
 2-functor $\mM_X(\beta)$ is an algebraic 
stack locally of finite type, 
though it is neither of finite type nor separated. 

For $m\in \mathbb{Z}$ and 
$\sigma=\sigma_{B, \omega} \in U(X)$, 
let
$v=(\beta, m) \in \Gamma_X$ and 
\begin{align}\label{sub:bn}
\mM_{\sigma}(v) \subset \mM_X(\beta)
\end{align}
be the substack of $\sigma$-semistable objects
$E \in \Coh_{\le 1}(X)$ satisfying 
\begin{align}\label{cond:E}
\ch(E)=v=(\beta, m). 
\end{align}
The stack $\mM_{\sigma}(v)$
is a finite type open substack of $\mM_X(\beta)$. 
Indeed, the stack $\mM_{\sigma}(v)$
is constructed as a GIT quotient stack (see~\cite[Lemma~7.4]{Todstack}), 
hence we have the 
projective coarse moduli space
$M_{\sigma}(v)$ together with the natural 
morphism
\begin{align}\label{p:natural}
p_M \colon \mM_{\sigma}(v) \to M_{\sigma}(v). 
\end{align}

For $\beta \in N_1(X)$, 
the Chow functor 
\begin{align}\label{funct:chow}
\cC how_X(\beta) \colon 
Sch^{red}/\mathbb{C} \to Set
\end{align}
is defined in~\cite{Ryd} by associating 
a reduced $\mathbb{C}$-scheme $S$ 
to the set of  
relative 
cycles on $X \times S$ over $S$, 
whose restriction to $X \times \{s\}$ for 
any closed point $s \in S$ is 
pure one dimensional with homology class 
$\beta$ (see~\cite[Section~4]{Ryd}). 
The functor (\ref{funct:chow})
is represented by a reduced projective scheme 
\begin{align*}
\Chow_X(\beta)
\end{align*}
 called \textit{Chow variety}, 
whose closed points correspond to 
effective one or zero cycles on $X$ with homology class $\beta$. 

Let $S$ be a reduced $\mathbb{C}$-scheme 
and $\eE \in \Coh(X \times S)$ be a 
$S$-valued point of $\mM_X(\beta)$. 
Then by~\cite[Theorem~7.14]{Ryd}
there is a canonical relative 
cycle on $X \times S$ whose support is $\Supp(\eE)$. 
It induces a morphism
\begin{align*}
S \to \Chow_X(\beta)
\end{align*}
which sends $s \in S$ to the 
fundamental cycle of $\eE_s$. The above morphism only 
depends on the isomorphism class of the sheaf $\eE$, thus 
induces the morphism of reduced stacks
\begin{align}\label{HC:stack}
\pi_{\mM} \colon \mM^{\rm{red}}_X(\beta) \to \Chow_X(\beta). 
\end{align}
The above morphism is called \textit{Hilbert-Chow (HC) map}. 
By restricting the above morphism 
to the open substack $\mM_{\sigma}(v) \subset \mM_X(\beta)$
for $\sigma \in U(X)$
and $v=(\beta, m) \in \Gamma_X$, 
we obtain the morphism
\begin{align}\label{HC:stack2}
\pi_{\mM} \colon
\mM_{\sigma}^{\rm{red}}(v)\to 
\Chow_X(\beta). 
\end{align}
By the universality of 
the coarse moduli space (see~\cite[Definition~2.2.1, Theorem~4.3.4]{MR1450870}), the morphism
 $\pi_{\mM}$ uniquely factors through 
the (reduced part of the) morphism $p_{\mM}$, 
where $p_M$ is the natural morphism (\ref{p:natural}). 
So we have the commutative diagram
\begin{align}\label{dia:chow}
\xymatrix{
\mM_{\sigma}^{\rm{red}}(v)  \ar@<-0.3ex>@{^{(}->}[r]  \ar[d]_-{p_M} &
\mM^{\rm{red}}_X(\beta) \ar[d]^-{\pi_{\mM}} \\
M_{\sigma}^{\rm{red}}(v) \ar[r]_-{\pi_M} & \Chow_X(\beta).
}
\end{align}

\subsection{$d$-critical schemes}
We recall the notion of $d$-critical schemes and 
$d$-critical
stacks introduced in~\cite{JoyceD}. 
For any complex scheme $T$, 
Joyce~\cite{JoyceD} shows that 
there exists a canonical sheaf of 
$\mathbb{C}$-vector spaces 
$\sS_{T}$ on $T$
satisfying the following
property:
for any Zariski open subset $R \subset T$
and a closed embedding 
$i \colon R \hookrightarrow V$
into a smooth scheme $V$, 
there is an exact sequence
\begin{align}\label{S:property}
0 \longrightarrow \sS_{T}|_{R}
 \longrightarrow \oO_V/I^2 \stackrel{d_{\rm{DR}}}{\longrightarrow}
\Omega_V/I \cdot \Omega_V. 
\end{align}
Here $I \subset \oO_V$ is the ideal sheaf
which defines $R$
and $d_{\rm{DR}}$ is the de-Rham differential. 
Moreover there is a natural decomposition 
\begin{align*}
\sS_T=\sS_T^0 \oplus \mathbb{C}_T
\end{align*}
where $\mathbb{C}_T$ is the constant sheaf
on $T$. 
The sheaf $\sS_T^{0}$ restricted to $R$ is 
the kernel of the 
composition 
\begin{align*}
\sS_T|_{R} \hookrightarrow \oO_V/I^2 \twoheadrightarrow \oO_{R^{\rm{red}}}. 
\end{align*}
For example, 
suppose that  
$f \colon V \to \mathbb{A}^1$ is
a regular function such that 
\begin{align}\label{R=df}
R=\{df=0\}, \ 
f|_{R^{\rm{red}}}=0.
\end{align}
Then we have 
$I=(df) \cneq \Imm(T_V \stackrel{df}{\to} \oO_{V})$ and 
$f+(df)^2$ is an element of
$\Gamma(R, \sS_{T}^0|_{R})$. 

\begin{defi}\emph{(\cite{JoyceD})}
A pair $(T, s)$
for a complex scheme $T$
and $s \in \Gamma(T, \sS_T^0)$
is called a $d$-critical 
scheme 
if for any 
$x \in T$, there is an open 
neighborhood $x \in R \subset T$,  
a closed embedding $i \colon R \hookrightarrow V$
into a smooth scheme $V$, 
a regular function $f \colon V \to \mathbb{A}^1$
satisfying (\ref{R=df})
such that 
$s|_{R}=f+(df)^2$
holds. 
In this case, 
the data
\begin{align}\label{crit:chart}
(R, V, f, i)
\end{align}
is called a $d$-critical chart.
The section $s$ is called a $d$-critical 
structure of $T$.  
\end{defi} 
Given a $d$-critical scheme $(T, s)$, there exists a line bundle
$K_{T, s}$ on $T^{\rm{red}}$
called \textit{virtual canonical line bundle}
(see~\cite[Section~2.4]{JoyceD} \footnote{In~\cite[Section~2.4]{JoyceD}, this
was just called canonical bundle.}),
such that for any
$d$-critical chart (\ref{crit:chart})
there is a natural isomorphism 
\begin{align}\label{nat:K0}
K_{T, s}|_{R^{\rm{red}}} \stackrel{\cong}{\to} K_V^{\otimes 2}|_{R^{\rm{red}}}. 
\end{align}
\begin{defi}\emph{(\cite{JoyceD})}
An orientation of a $d$-critical scheme $(T, s)$
is a
 choice of a square root line bundle 
$K_{T, s}^{1/2}$ for $K_{T, s}$ on $T^{\rm{red}}$
and an isomorphism
\begin{align}\label{isom:orient}
(K_{T, s}^{1/2})^{\otimes 2} \stackrel{\cong}{\to}
K_{T, s}. 
\end{align}
A $d$-critical scheme with an orientation is called
an oriented $d$-critical scheme. 
\end{defi}
\subsection{$d$-critical stacks}
Let $\mM$ be an algebraic stack over $\mathbb{C}$. 
The category
 of sheaves of $\mathbb{C}$-vector spaces on 
$\mM$ is defined in the lisse-\'etale site of $\mM$. 
This is equivalent to the category 
$\mathrm{Sh}(\mM)$ defined as follows (see~\cite{GL} for details):
an object $\fF$ of $\mathrm{Sh}(\mM)$
consists of data 

(i) 
For each 
$\mathbb{C}$-scheme $T$ and a smooth 1-morphism 
$t \colon T \to \mM$, we are given a sheaf of 
$\mathbb{C}$-vector spaces $\fF(T, t)$
on $T$ in \'etale topology. 

(ii) For 
$\mathbb{C}$-schemes $T$, $U$
and smooth 1-morphisms 
$t \colon T \to \mM$, $u \colon U \to \mM$
with a 2-commutative diagram
\begin{align*}
\xymatrix{
T \ar[dr]_{t}  \ar[rr]^{\phi} & &  U \ar[ld]^{u} \\
   &   \mM  & 
}
\end{align*}
we are given a morphism 
\begin{align}\label{sheaf:rest}
\phi^{-1} \fF(U, u) \to \fF(T, t)
\end{align}
of sheaves of $\mathbb{C}$-vector spaces on 
$T$ in \'etale topology. 

The above data should satisfy several compatibility conditions. 
A global section $s \in H^0(\fF)$ of an object 
$\fF \in \mathrm{Sh}(\mM)$ consists of 
global sections $s(T, t) \in H^0(\fF(T, t))$
for each $\mathbb{C}$-scheme $T$ and a smooth 
morphism $t \colon T \to \mM$, 
such that the morphism (\ref{sheaf:rest})
sends $s(U, u)$ to $s(T, t)$. 

By~\cite[Corollary~2.52]{JoyceD}, there is a canonical 
sheaf of $\mathbb{C}$-vector spaces 
$\sS_{\mM}^0$ on 
an algebraic stack $\mM$, 
such that for any 
scheme $T$ and a smooth 1-morphism 
$t \colon T \to \mM$ we have
\begin{align*}
\sS_{\mM}^0(T, t)=\sS_{T}^0. 
\end{align*}
\begin{defi}\emph{(\cite{JoyceD})}\label{def:dstack}
A pair $(\mM, s)$ for an algebraic stack $\mM$ over 
$\mathbb{C}$ and a global section $s \in H^0(\sS_{\mM}^0)$ is called 
a $d$-critical stack if for any $\mathbb{C}$-scheme $T$ and 
a smooth 1-morphism $t \colon T \to \mM$, 
the pair $(T, s(T, t))$
is a $d$-critical scheme. 
\end{defi}

Given a $d$-critical stack $(\mM, s)$, there exists a 
line bundle $K_{\mM, s}$ on $\mM^{\rm{red}}$, called
\textit{virtual canonical line bundle}, such that 
for any $\mathbb{C}$-scheme
$T$ and a smooth 1-morphism $t \colon T \to \mM$, 
so that $t^{\rm{red}} \colon T^{\rm{red}} \to \mM^{\rm{red}}$ is 
also smooth, there is a natural isomorphism
\begin{align}\label{nat:kstack}
K_{\mM, s}(T^{\rm{red}}, t^{\rm{red}}) \stackrel{\cong}{\to}
K_{T, s(T, t)} \otimes \mathrm{det}(\Omega_{T/\mM}|_{T^{\rm{red}}})^{\otimes -2}. 
\end{align}
Here $(T, s(T, t))$ is the $d$-critical scheme in Definition~\ref{def:dstack}. 

\begin{defi}\emph{(\cite{JoyceD})}
An orientation of a $d$-critical stack $(\mM, s)$
is a
 choice of a square root line bundle 
$K_{\mM, s}^{1/2}$ for $K_{\mM, s}$ on $\mM^{\rm{red}}$
and an isomorphism
\begin{align}\label{isom:orient2}
(K_{\mM, s}^{1/2})^{\otimes 2} \stackrel{\cong}{\to}
K_{\mM, s}. 
\end{align}
A $d$-critical stack with an orientation is called
an oriented $d$-critical stack. 
\end{defi}
For an oriented $d$-critical stack $(\mM, s, K_{\mM, s}^{1/2})$, 
let $(T, t)$ be as in Definition~\ref{def:dstack}. 
Then we have the line bundle on $T^{\rm{red}}$
\begin{align*}
K_{T, s(T, t)}^{1/2}=K_{\mM, s}^{1/2}(T^{\rm{red}}, t^{\rm{red}}) \otimes \mathrm{det}(\Omega_{T/\mM}|_{T^{\rm{red}}}). 
\end{align*}
Then the isomorphism (\ref{nat:kstack}), 
induces the isomorphism
\begin{align}\label{induce:orient}
(K_{T, s(T, t)}^{1/2})^{\otimes 2} \stackrel{\cong}{\to}
K_{T, s(T, t)}
\end{align}
which gives an orientation of the $d$-critical 
scheme $(T, s(T, t))$. 

As mentioned in~\cite{JoyceD}, 
the above 
notions of $d$-critical structures on schemes 
and stacks are
naturally extended to those
of analytic $d$-critical structures
for complex analytic spaces and stacks
respectively. 
Moreover given an algebraic $d$-critical 
structure on an algebraic stack $\mM$, 
it naturally gives an analytic 
$d$-critical structure on the 
analytification of $\mM$.

\begin{exam}\label{exam:dstack}
Suppose that an algebraic $\mathbb{C}$-group $G$ acts on 
a complex analytic space $R$, and 
set $\mM=[R/G]$. 
Then as in~\cite[Example~2.55]{JoyceD}, 
we have $H^0(\sS_{\mM}^0)=H^0(\sS_{R}^0)^G$. 
Let $V$ be a complex manifold with $G$-action, 
$f \colon V \to \mathbb{C}$ be a $G$-invariant 
analytic function such that $R=\{df=0\}$. 
Then we have
\begin{align*}
s=f+(df)^2 \in H^0(\sS_R^0)^G. 
\end{align*}
The pair $(\mM, s)$ is an example of an analytic 
 $d$-critical stack. 
\end{exam}

\subsection{Perverse sheaves of vanishing cycles}
Let $f \colon V \to \mathbb{C}$ be 
a holomorphic function on a complex manifold $V$, 
and set 
$R=\{df=0\}$.
Suppose that $f|_{R^{\rm{red}}} =0$
and set $V_0=f^{-1}(0)$.  
We have the associated vanishing cycle functor
(see~\cite[Theorem~5.2.21]{Dimbook})
\begin{align*}
\phi_f \colon \Perv(V) \to \Perv(V_0). 
\end{align*}
Let $\IC(V) \in \Perv(V)$ be the intersection complex 
on $V$, which 
coincides with $\mathbb{Q}_V[\dim V]$ 
since $V$ is smooth. 
We have the perverse sheaf 
of vanishing cycles
supported on $R^{\rm{red}} \subset V_0$
\begin{align}\label{vsheaf}
\phi_f(\IC(V)) \in \Perv(R) \subset \Perv(V_0). 
\end{align}
Let $(T, s)$ be a $d$-critical scheme or 
$d$-critical analytic space. 
For a $d$-critical chart 
$(R, V, f, i)$ as in 
(\ref{crit:chart}), 
we have the perverse sheaf of vanishing cycles 
(\ref{vsheaf}) on $R$. 
In~\cite{MR3353002} it is proved that 
if $(T, s)$ is oriented, then the perverse
sheaves of vanishing cycles (\ref{vsheaf}) glue 
to give a global perverse sheaf on $T$. 
Let
\begin{align}\label{isom:KK}
(K_{T, s}^{1/2}|_{R^{\rm{red}}})^{\otimes 2}
\cong K_V^{\otimes 2}|_{R^{\rm{red}}}
\end{align}
be the isomorphism given by
the composition of (\ref{nat:K0})
and (\ref{isom:orient}). 
Then there is a 
$\mathbb{Z}/2\mathbb{Z}$-principal bundle 
$
\tau_R \colon 
\widetilde{R}^{\rm{red}} \to R^{\rm{red}}
$
which parametrizes local square roots 
of the isomorphism (\ref{isom:KK}). 
We have the decomposition
\begin{align*}
\tau_{R\ast} \mathbb{Q}_{\widetilde{R}^{\rm{red}}}
=\mathbb{Q}_{R^{\rm{red}}} \oplus \lL
\end{align*}
for a rank one local system $\lL$ on $R^{\rm{red}}$. 
The following result is proved in~\cite{MR3353002}
(also see~\cite{KL} for the similar result 
in the framework of virtual critical structures): 
\begin{thm}\emph{(\cite[Theorem~6.9]{MR3353002}, \cite[Theorem~4.12]{BBBJ})}\label{thm:MR3353002}
\begin{enumerate}
\item
For an oriented 
$d$-critical scheme or an oriented $d$-critical 
analytic space $(T, s, K_{T, s}^{1/2})$, 
there exists a natural 
perverse sheaf 
$\phi_{T}$ on $T$ such that for any 
$d$-critical chart (\ref{crit:chart})
there is a natural isomorphism\footnote{In~\cite[Theorem~6.9]{MR3353002}, 
the tensor product in the right hand side of (\ref{isom:IC}) is denoted as 
$\otimes_{\mathbb{Z}/2\mathbb{Z}}\widetilde{R}^{\rm{red}}$. 
The latter tensor product is defined as the one given in the right hand 
side of (\ref{isom:IC}) in~\cite[Definition~2.9]{MR3353002}.
}
\begin{align}\label{isom:IC}
\phi_{T}|_{R} \stackrel{\cong}{\to}
\phi_f(\IC(V)) \otimes \lL. 
\end{align}
\item
For an oriented $d$-critical algebraic or analytic stack 
$(\mM, s, K_{\mM, s}^{1/2})$, 
there exists a natural perverse sheaf $\phi_{\mM}$ on 
$\mM$ such that for any $(T, t)$ as in Definition~\ref{def:dstack}
we have
\begin{align*}
\phi_{\mM}(T, t)=\phi_T[-d_t]. 
\end{align*}
Here $d_t$ is the relative dimension of 
$t \colon T \to \mM$
and $\phi_T$ is the perverse sheaf in (i) 
for
$(T, s(T, t))$ with orientation given by (\ref{induce:orient}). 
\end{enumerate}
\end{thm}

\subsection{CY properties of $d$-critical stacks}
Let $X$ be a smooth projective CY 3-fold.
For $\beta \in N_1(X)$, let  
 $\mM_X(\beta)$ be
 the moduli stack defined as in (\ref{moduli:2funct}). 
By the result of~\cite{BBBJ},
we have the following: 
\begin{thm}\emph{(\cite{BBBJ})}\label{thm:CYdcrit}
There is a canonical $d$-critical stack structure 
$s \in H^0(\sS_{\mM}^0)$ on $\mM_X(\beta)$, whose 
virtual canonical line bundle is given by
\begin{align}\label{vir:K}
K_{\mM_X(\beta), s}=
K_{\mM_X(\beta)}^{\rm{vir}} \cneq
\det \dR \hH om_{pr_{\mM}}(\eE, \eE). 
\end{align} 
Here $\eE$ is the universal family on $X \times \mM_{X}(\beta)$ and 
$pr_{\mM} \colon X \times \mM_X(\beta) \to \mM_X(\beta)$ is the projection. 
\end{thm}

We next consider CY conditions on our moduli stacks of 
one dimensional sheaves. 
In general, we introduce the following definition:
\begin{defi}\label{def:vir}
Let $(\mM, s)$ be a $d$-critical stack with a 
morphism 
\begin{align*}
\pi_{\mM} \colon \mM^{\rm{red}} \to \Chow_X(\beta)
\end{align*}
for $\beta \in N_1(X)$. 
We say $\mM$ is Calabi-Yau (CY)
at $\gamma \in \Chow_X(\beta)$ if 
there is an analytic open neighborhood 
$\gamma \in U \subset \Chow_X(\beta)$
such that $K_{\mM, s}$ is trivial 
on $\pi_{\mM}^{-1}(U)$. 
\end{defi}

We propose the following conjecture: 
\begin{conj}\label{conj:vir}
The stack $\mM_X(\beta)$ is CY 
at any point $\gamma \in \Chow_X(\beta)$
for the HC map (\ref{HC:stack}). 
\end{conj}

Note that the $d$-critical structure on $\mM_X(\beta)$ in 
Theorem~\ref{thm:CYdcrit} 
induces the one on $\mM_{\sigma}(v)$ by the open embedding (\ref{sub:bn}).
Moreover if Conjecture~\ref{conj:vir} is true, then 
the open substack $\mM_{\sigma}(v) \subset \mM_X(\beta)$
is also CY at $\gamma \in \Chow_X(\beta)$
for the HC map (\ref{HC:stack2}). 

 Suppose that $\mM_{\sigma}(v)$ is CY at 
$\gamma \in \Chow_X(\beta)$, and take 
an open subset $\gamma \in U \subset \Chow_X(\beta)$
as in Definition~\ref{def:vir}. 
Below 
we denote the pull-back of the diagram 
(\ref{dia:chow}) 
to $U$ by (omitting `red' for simplicity) 
\begin{align}\label{dia:chow2}
\xymatrix{
\mM_{\sigma}(v)|_{U} \ar@<-0.3ex>@{^{(}->}[r]  \ar[d]_{p_M} &
\mM_X(\beta)|_{U} \ar[d]^{\pi_{\mM}} \\
M_{\sigma}(v)|_{U} \ar[r]_{\pi_M} & U.
}
\end{align}
By the CY condition of $\mM_{\sigma}(v)$, 
there is an orientation data 
of $\mM_{\sigma}(v)|_{U}$ satisfying
\begin{align}\label{isom:half}
(K_{\mM_{\sigma}(v)|_{U}}^{\rm{vir}})^{1/2}
\cong \oO_{\mM_{\sigma}(v)|_{U}}
\end{align}
as line bundles on the stack $\mM_{\sigma}(v)|_{U}$. 
Such an orientation data is called 
a \textit{Calabi-Yau (CY) orientation data}. 
Given a CY orientation data of $\mM_{\sigma}(v)|_{U}$, 
by Theorem~\ref{thm:MR3353002} (ii) 
we have the 
associated perverse sheaf of vanishing cycles 
\begin{align}\label{per:MU}
\phi_{\mM_{\sigma}(v)|_{U}} \in \Perv(\mM_{\sigma}(v)|_{U}). 
\end{align}
\subsection{Definition of generalized GV invariants}
We keep the situation and notation in the previous 
subsection. 
Using the perverse sheaf (\ref{per:MU}), we 
give the following definition:
\begin{defi}\label{def:def:phiM}
For $v=(\beta, m) \in \Gamma_X$
and $\sigma \in U(X)$, 
suppose that the $d$-critical 
stack $\mM_{\sigma}(v)$
is CY at $\gamma \in \Chow_X(\beta)$. 
Then
we define the perverse sheaf
$\phi_{M_{\sigma}(v)|_{U}}$ on $M_{\sigma}(v)|_{U}$ 
in the diagram (\ref{dia:chow2}) by 
\begin{align}\label{def:phiM}
\phi_{M_{\sigma}(v)|_{U}}\cneq 
\pH^1(\dR p_{M\ast} \phi_{\mM_{\sigma}(v)|_{U}})
\in \Perv(M_{\sigma}(v)|_{U}). 
\end{align}
\end{defi}
As we will see in the next subsection, 
the perverse 
sheaf (\ref{dia:chow2}) 
is an analogue of BPS sheaves introduced in~\cite{DaMe}
for representations of quivers with super-potentials. 
The following lemma implies that (\ref{dia:chow2})
is a generalization of the 
perverse sheaf in Theorem~\ref{thm:MR3353002} (i), and explains 
the reason of taking the first perverse cohomology 
(which is due to 
the shift convention of perverse sheaves by $\dim B\mathbb{C}^{\ast}=-1$): 
\begin{lem}\label{rmk:gerbe}
Suppose that $\mM_{\sigma}(v)|_{U}$ is a trivial 
$\mathbb{C}^{\ast}$-gerbe over $M_{\sigma}(v)|_{U}$, i.e. 
$\mM_{\sigma}(v)|_{U}=M_{\sigma}(v)|_{U} \times B\mathbb{C}^{\ast}$, 
e.g. the case of $\sigma=i\omega$ and
$v=(\beta, 1)$. 
In this case, $M_{\sigma}(v)|_{U}$ itself is a $d$-critical scheme, 
so we have the associated perverse sheaf
$\phi_{M_{\sigma}(v)|_{U}}'$ on $M_{\sigma}(v)|_{U}$
defined in Theorem~\ref{thm:MR3353002} (i)
using a CY orientation data. 
In this case, we have
\begin{align}\label{id:phi}
\phi_{M_{\sigma}(v)|_{U}}=\phi_{M_{\sigma}(v)|_{U}}'.
\end{align}
\end{lem}
\begin{proof}
 The perverse sheaf $\phi_{\mM_{\sigma}(v)|_{U}}$
regarded as a $\mathbb{C}^{\ast}$-equivariant 
perverse sheaf on 
$M_{\sigma}(v)|_{U}$ coincides with 
$\phi_{M_{\sigma}(v)|_{U}}'[-1]$, 
because $\dim B \mathbb{C}^{\ast}=-1$. 
Since $H^{\ast}(B\mathbb{C}^{\ast}, \mathbb{Q})=\mathbb{Q}[t]$
where $t$ is of degree two, we have 
\begin{align*}
\dR p_{M\ast}\phi_{\mM_{\sigma}(v)|_{U}} \cong 
\phi_{M_{\sigma}(v)|_{U}}'[-1] \otimes \mathbb{Q}[t].
\end{align*}
By taking the first perverse cohomologies, 
we obtain (\ref{id:phi}). 
\end{proof}

We then introduce our invariant 
$\Phi_{\sigma}(\gamma, m)$ as follows (see the diagram (\ref{dia:chow2})): 
\begin{defi}\label{def:def:phiM2}
In the situation of Definition~\ref{def:def:phiM}, 
we define the following Laurent polynomial in $y$
\begin{align}\label{def:Phi:sigma}
\Phi_{\sigma}(\gamma, m) \cneq 
\sum_{i \in \mathbb{Z}}
\chi(\pH^i(\dR \pi_{M\ast}\phi_{M_{\sigma}(v)|_{U}})|_{\gamma})y^i
\in \mathbb{Z}[y^{\pm 1}].
\end{align}
\end{defi}

We have the following lemma:
\begin{lem}\label{lem:inde}
The Laurent polynomial 
$\Phi_{\sigma}(\gamma, m)$ is 
independent of a choice of 
a CY orientation data
of $\mM_{\sigma}(v)|_{U}$. 
\end{lem}
\begin{proof}
The proof is similar to~\cite[Lemma~2.7]{MT}. 
By fixing an isomorphism (\ref{isom:half}), 
a CY orientation data of 
$\mM_{\sigma}(v)|_{U}$ is regarded as an isomorphism
\begin{align}\label{isom:theta}
\vartheta \colon 
\oO_{\mM_{\sigma}(v)|_{U}} \otimes_{\oO_{\mM_{\sigma}(v)|_{U}}} \oO_{\mM_{\sigma}(v)|_{U}} \stackrel{\cong}{\to}
\oO_{\mM_{\sigma}(v)|_{U}}. 
\end{align}
We have the natural isomorphism
\begin{align*}
\oO_{\mM_{\sigma}(v)|_{U}} \otimes_{\oO_{\mM_{\sigma}(v)|_{U}}} \oO_{\mM_{\sigma}(v)|_{U}} \stackrel{\cong}{\to}
\oO_{\mM_{\sigma}(v)|_{U}}, \ x \otimes y \mapsto xy. 
\end{align*}
Therefore an isomorphism (\ref{isom:theta})
is regarded as an invertible element
$\vartheta \in H^0(\oO_{\mM_{\sigma}(v)|_{U}})$. 

Let us take two invertible elements 
$\vartheta^{(1)}, \vartheta^{(2)} \in H^0(\oO_{\mM_{\sigma}(v)|_{U}})$, 
and 
\begin{align*}
\phi_{\mM_{\sigma}(v)|_{U}}^{(i)} \in \Perv(\mM_{\sigma}(v)|_{U}), \ 
\phi_{M_{\sigma}(v)|_{U}}^{(i)} \in \Perv(M_{\sigma}(v)|_{U})
\end{align*}
the associated perverse sheaves (\ref{per:MU}), (\ref{def:phiM})
w.r.t. the orientation data $\vartheta^{(i)}$
respectively. 
We consider the following commutative diagram 
\begin{align*}
\xymatrix{
\mM_{\sigma}(v)|_{U}
 \ar[r]^{p_M}\ar[dr]_{\overline{\pi}_{\mM}} & 
M_{\sigma}(v)|_{U} \ar[d]^{\overline{\pi}_M}
\ar[dr]^{\pi_M} & \\
&  B \ar[r]_{\pi_B}  &  U. 
}
\end{align*}
Here
$\overline{\pi}_M$ is the Stein factorization of 
$\pi_M$ and 
 $\pi_{\mM}=\pi_M \circ p_M=\pi_B \circ \overline{\pi}_M$.
 Since we have 
\begin{align*}
p_{M\ast}\oO_{\mM_{\sigma}(v)|_{U}}=\oO_{M_{\sigma}(v)|_{U}}, \ 
\overline{\pi}_{M\ast}\oO_{M_{\sigma}(v)|_{U}}=\oO_{B}, 
\end{align*}
we have
$\overline{\pi}_{\mM\ast}\oO_{\mM_{\sigma, U}}(v)=\oO_{B}$.
Therefore the element
\begin{align*}
\psi \cneq \vartheta^{(1)} \circ (\vartheta^{(2)})^{-1}
\in H^0(\oO_{\mM_{\sigma}(v)|_{U}})
\end{align*}
is written as 
$\psi=\overline{\pi}_{\mM}^{\ast}\overline{\psi}$ for 
some invertible 
element $\overline{\psi} \in H^0(\oO_B)$. 
Let 
\begin{align*}
\iota \colon \widetilde{U} \to B
\end{align*}
 be the 
$\mu_2$-torsor given by local square roots of $\overline{\psi}$, 
and $\lL_B$ a rank one local system on $B$
given by 
$\iota_{\ast}\mathbb{Q}_{\widetilde{U}}
=\mathbb{Q}_{B} \oplus \lL_B$. 
Then by (\ref{isom:IC}), we have 
\begin{align*}
\phi_{\mM_{\sigma}(v)|_{U}}^{(2)}=
\phi_{\mM_{\sigma}(v)|_{U}}^{(1)}
 \otimes \overline{\pi}_{\mM}^{\ast}\lL_B, \ 
\phi_{M_{\sigma}(v)|_{U}}^{(2)}=
\phi_{M_{\sigma}(v)|_{U}}^{(1)} \otimes \overline{\pi}_M^{\ast}\lL_B. 
\end{align*}
Since $\pi_B$ is a finite map, $\pi_{B\ast}$ preserves the 
perverse t-structures. 
It follows that
\begin{align*}
\pH^i(\dR \pi_{M\ast}\phi_{M_{\sigma}(v)|_{U}}^{(2)})=
\pH^i(\pi_{B\ast}(\dR \overline{\pi}_{M\ast} \phi_{M_{\sigma}(v)|_{U}}^{(1)}
\otimes \lL_B)). 
\end{align*}
Since $\lL_B$ is a rank one local system, 
we have
\begin{align*}
\chi(\pH^i(\pi_{B\ast}(\dR \overline{\pi}_{M\ast} \phi_{M_{\sigma}(v)|_{U}}^{(1)}\otimes \lL_B))|_{\gamma})=
\chi(\pH^i(\dR \pi_{M\ast}\phi_{M_{\sigma}(v)|_{U}}^{(1)}|_{\gamma}). 
\end{align*}
Therefore the lemma follows. 
\end{proof}

The following is the main conjecture we address in this paper. 
\begin{conj}\label{conj:main}
For $v=(\beta, m) \in \Gamma_X$ with 
$\beta \neq 0$ and $\sigma \in U(X)$, suppose that 
$\mM_{\sigma}(v)$ is CY at $\gamma \in \Chow_X(\beta)$. 
Then the Laurent polynomial $\Phi_{\sigma}(\gamma, m)$
is independent of $\sigma$ and $m$. 
\end{conj}
By Lemma~\ref{rmk:gerbe}, 
the local GV invariants $n_{g, \gamma}$ defined in~\cite{MT} is characterized 
by the identity
\begin{align}\label{id:GV1}
\Phi_{\sigma=i\omega}(\gamma, 1)=\sum_{g\ge 0} n_{g, \gamma}(y^{1/2}+y^{-1/2})^{2g}.  
\end{align}
Suppose that Conjecture~\ref{conj:main} holds. 
Then $\Phi_{\sigma}(\gamma, m)=\Phi_{\sigma=i\omega}(\gamma, 1)$, so 
if we define $n_{g, \gamma, m}(\sigma) \in \mathbb{Z}$ by the identity
\begin{align}\label{id:GV2}
\Phi_{\sigma}(\gamma, m)=\sum_{g\ge 0} n_{g, \gamma, m}(\sigma)(y^{1/2}+y^{-1/2})^{2g}
\end{align}
then $n_{g, \gamma, m}(\sigma)$ is independent of 
$(\sigma, m)$. 

\begin{rmk}
In the case of $\sigma=i\omega$ and $m=1$, 
the perverse sheaf $\phi_{M_{\sigma}(v)|_{U}}'$ in 
Lemma~\ref{rmk:gerbe} is known to be self dual, so 
$\Phi_{\sigma=i\omega}(\gamma, 1)$ is 
invariant under $y \mapsto 1/y$. 
However in other cases, 
it is not clear whether
$\Phi_{\sigma}(\gamma, m)$ is invariant under
$y \mapsto 1/y$.
So a priori, 
 we cannot define $n_{g, \gamma, m}(\sigma)$
as in (\ref{id:GV2})
without assuming Conjecture~\ref{conj:main}.
\end{rmk}

\begin{rmk}\label{rmk:weaker}
In Definition~\ref{def:def:phiM2}, it was enough 
to assume that the stack $\mM_{\sigma}(v)$ is CY at $\gamma$
to define the invariant $\Phi_{\sigma}(\gamma, m)$.  
The CY condition of the bigger stack 
$\mM_X(\beta)$ in Conjecture~\ref{conj:vir} will play a role 
in showing the independence of $\sigma$ of the invariant
 (\ref{def:Phi:sigma})
(see the proof of Theorem~\ref{thm:inde}). 
\end{rmk}

\subsection{Relation to BPS sheaves}\label{subsec:BPSsheaf}
The definition of the perverse sheaf (\ref{def:phiM})
was suggested to the author by Ben Davison, 
as it is an analogue 
of BPS sheaves 
introduced in~\cite{DaMe}
in the case of representations of quivers with super-potentials.
Here we explain its relationship. 
Let $Q$ be a symmetric quiver (see Definition~\ref{def:symmetric} below)
with super-potential $W$. 
Let $\mM_{Q}$, 
$\mM_{(Q, \partial W)}$ be the moduli stack of finite dimensional 
$Q$-representations, 
$(Q, W)$-representations
respectively,  
and $M_{Q}$, $M_{(Q, \partial W)}$ their coarse moduli spaces. 
We have the commutative diagram
\begin{align*}
\xymatrix{
\mM_{(Q, \partial W)} \ar[d]_-{p_{(Q, \partial W)}}
\ar@<-0.3ex>@{^{(}->}[r] & \mM_{Q} \ar[d]_-{p_Q} \ar[rd]^-{\tr W} &  \\
M_{(Q, \partial W)}
\ar@<-0.3ex>@{^{(}->}[r] & M_Q \ar[r]_-{\tr \overline{W}} & \mathbb{C}.
}
\end{align*}
Here the horizontal arrows are closed immersions, 
the vertical arrows are natural morphisms to the coarse 
moduli spaces. 
The function 
$\tr W$ is defined from the super-potential $W$ 
as in Subsection~\ref{subsec:conv} below, 
and $\tr \overline{W}$ is the induced function 
by $p_{Q\ast}\oO_{\mM_Q}=\oO_{M_Q}$. 
We have 
\begin{align*}
\mM_{(Q, \partial W)}=\{d (\tr W)=0\}.
\end{align*}
Since $\mM_Q$ is a smooth stack, 
the stack $\mM_{(Q, \partial W)}$ has a $d$-critical 
structure whose virtual canonical line bundle is
trivial by Proposition~\ref{prop:vir:tri} below. 
Using a CY orientation data, we 
obtain the associated perverse sheaf on the stack 
$\mM_{(Q, \partial W)}$
\begin{align*}
\phi_{\mM_{(Q, \partial W)}} \in \Perv(\mM_{(Q, \partial W)}). 
\end{align*}
Then by~\cite[Theorem~4.7]{DaMe}, we have
\begin{align}\label{id:BPS}
\pH^1(\dR p_{(Q, \partial W)\ast}\phi_{\mM_{(Q, \partial W)}})=
\phi_{\tr \overline{W}}(j_{!\ast}^s\IC_{M_Q^s}). 
\end{align}
Here $j^s \colon M_Q^s \subset M_Q$ is the 
open immersion of the simple 
part.
The RHS of (\ref{id:BPS})
is supported on $M_{(Q, \partial W)}$, and defined 
to be the BPS sheaf in~\cite{DaMe}. 
The identity (\ref{id:BPS})
explains that the perverse sheaf (\ref{def:phiM})
is an analogue of the BPS sheaf introduced in~\cite{DaMe}. 

In Remark~\ref{rmk:BPS}, 
we will return to this point of view 
and see that (under some matching of $d$-critical structures) 
the perverse sheaf (\ref{def:phiM}) is analytic locally the BPS 
sheaf for the representations of a quiver 
with a (formal but convergent) super-potential. 

\section{Examples}\label{sec:exam}
In this section, we compute the invariants 
$\Phi_{\sigma}(\gamma, m)$ in some cases. 
\subsection{The case of zero-cycles}
In this subsection, we compute 
$\Phi_{\sigma}(\gamma, m)$ when $\gamma$ is a zero cycle, i.e.
$\gamma \in S^m(X)$. 
For a (not necessary projective) CY 3-fold $X$, 
let $\mM_0(m)$ be the moduli stack of 
zero dimensional coherent sheaves on $X$ with length $m$. 
Since any zero dimensional sheaf has a compact support, 
the stack $\mM_0(m)$ has a $d$-critical structure\footnote{Although $\mM_0(m)$ is non-compact, 
its derived enhancement admits a $(-1)$-shifted symplectic structure~\cite{Preygel} and the associated 
$d$-critical 
structure by~\cite[Theorem~1.2]{BBBJ}.}
with the natural morphism to the coarse moduli space
\begin{align}\label{HC:zero}
p \colon \mM_0^{\rm{red}}(m) \to S^m(X). 
\end{align}
If $X$ is furthermore projective, 
$\mM_0(m)=\mM_{\sigma}(v)$
for $v=(0, m)$
and any $\sigma \in U(X)$. 

We first show the CY condition of $\mM_0(m)$
as in Definition~\ref{def:vir}. 
Indeed in this case, we have the following
global CY property:
\begin{prop}\label{lem:M0vir}
For any CY 3-fold $X$, the 
virtual canonical line bundle 
$K_{\mM_0(m)}^{\rm{vir}}$ of the $d$-critical stack 
$\mM_0(m)$ is trivial, i.e. 
$K_{\mM_0(m)}^{\rm{vir}} \cong \oO_{\mM_0^{\rm{red}}(m)}$. 
\end{prop}
\begin{proof}
For the morphism (\ref{HC:zero}), 
we have the natural morphism
\begin{align}\label{nat:K}
p^{\ast}p_{\ast}K^{\rm{vir}}_{\mM_0(m)} \to K^{\rm{vir}}_{\mM_0(m)}. 
\end{align}
We show that 
$p_{\ast}K^{\rm{vir}}_{\mM_0(m)}$ is a line bundle on $S^m(X)$ and 
the morphism (\ref{nat:K}) is an isomorphism of line bundles on 
$\mM_0^{\rm{red}}(m)$. 
It is enough to show these claims analytic locally\footnote{
Since the virtual canonical line bundle is determined by a universal 
sheaf (see Theorem~\ref{thm:CYdcrit}), 	
it is enough to have analytic 
local isomorphism on $S^m(X)$, i.e. don't have to compare these d-critical structures.
} on $S^m(X)$, so 
we may assume that $X=\mathbb{C}^3$. 
In this case, the stack $\mM_0(m)$ is the moduli stack of 
$m$-dimensional representations of a quiver with a super-potential
\begin{align}\label{quiver:3}
(Q_{(3)}, W_{(3)})
\end{align}
where the quiver $Q_{(3)}$ has 
one vertex, three loops $A, B, C$, 
and the  
super-potential $W_{(3)}$ is given by 
$W_{(3)}=A[B, C]$. 
Therefore by Proposition~\ref{prop:vir:tri} below,  
we see that 
$K^{\rm{vir}}_{\mM_0(m)}$ is a trivial line bundle
when $X=\mathbb{C}^3$. 
Since $S^m(X)$ is the coarse moduli space of $\mM_0^{\rm{red}}(m)$, 
we have $p_{\ast}\oO_{\mM_0^{\rm{red}}(m)} \cong \oO_{S^m(X)}$, 
therefore the isomorphism of
 (\ref{nat:K}) follows when $X=\mathbb{C}^3$. 
Hence (\ref{nat:K}) is an isomorphism for any CY 3-fold $X$. 

Next we show that $K_{\mM_0(m)}^{\rm{vir}}$ is 
trivial for any CY 3-fold $X$. 
We consider the following morphisms
\begin{align*}
\xymatrix{
[X^{\times m}/S_m] 
\ar[r]^-{r} \ar[dr]_-{q} & \mM_0^{\rm{red}}(m) \ar[d]^-{p} \\
& S^m(X).
} 
\end{align*}
Here $S_m$ acts on $X^{\times m}$ by permutation, 
$r$ sends $(x_1, \ldots, x_m)$ to 
$\oplus_{i=1}^m \oO_{x_i}$ and $q$ is the natural morphism 
to the coarse moduli space. 
We will show that $r^{\ast}K_{\mM_0(m)}^{\rm{vir}}$ is a trivial 
line bundle on $[X^{\times m}/S_m]$. 
If this is true, then by pulling the isomorphism (\ref{nat:K}) back by $r$
we have the isomorphism
\begin{align*}
q^{\ast}p_{\ast}K_{\mM_0(m)}^{\rm{vir}} \stackrel{\cong}{\to}
\oO_{[X^{\times m}/S_m]}.
\end{align*}
Since $q_{\ast}\oO_{[X^{\times m}/S_m]} \cong 
\oO_{S^m(X)}$, 
by pushing forward the above isomorphism to $S^m(X)$
we obtain $p_{\ast}K_{\mM_0(m)}^{\rm{vir}} \cong \oO_{S^m(X)}$. 
Therefore $K_{\mM_0(m)}^{\rm{vir}}$ is trivial
 by the isomorphism(\ref{nat:K}). 

Let $p_i \colon X^{\times m} \to X$ be the projection to the $i$-th component, 
and set $\Delta_i=(\id_X \times p_i)^{\ast}\Delta$
for $\id_X \times p_i \colon X \times X^{\times m} \to X \times X$ and $\Delta$
is the diagonal in $X \times X$. 
The line bundle $r^{\ast}K_{\mM_0(m)}^{\rm{vir}}$ is a $S_m$-equivariant line 
bundle on $X^{\times m}$ given by 
\begin{align}\notag
r^{\ast}K_{\mM_0(m)}^{\rm{vir}} &=
\det \dR \hH om_{pr} \left(\oplus_{i=1}^m \oO_{\Delta_i}, \oplus_{i=1}^m \oO_{\Delta_i} \right) \\
\label{r*K}
& = \bigotimes_{1\le i, j \le m} L_{ij}. 
\end{align}
Here $pr \colon X \times X^{\times m} \to X^{\times m}$ is the projection, 
and $L_{ij}$ is the line bundle
\begin{align*}
L_{ij} \cneq \det \dR \hH om_{pr} \left(\oO_{\Delta_i}, \oO_{\Delta_j} \right).
\end{align*}
It is easy to see that the line bundle $L_{ij}$ is 
a trivial line bundle on $X^{\times m}$ (either by direct calculation using Hochschild resolution of the diagonal, or 
using a dimension filtration of K-theory as in~\cite[Proposition~3.13]{MT}). 
Let $e_{ij}$ be a no-where vanishing global section of $L_{ij}$. 
Then 
$\sigma \in S_m$ acts on the trivial line bundle (\ref{r*K})
by sending 
the basis $\prod_{i, j} e_{ij}$ to 
$\prod_{i, j} e_{\sigma(i)\sigma(j)}=\prod_{i, j}e_{ij}$. 
Thus 
the $S_m$-equivariant structure of the trivial line bundle 
(\ref{r*K}) is also trivial, and the proposition holds. 
\end{proof}
By Proposition~\ref{lem:M0vir}, 
we have a global orientation data of $\mM_0(m)$
satisfying
\begin{align*}
(K_{\mM_{0}(m)}^{\rm{vir}})^{1/2}
 \cong \oO_{\mM_{0}^{\rm{red}}(m)}. 
\end{align*}
Using the above orientation data, 
by Theorem~\ref{thm:MR3353002} we 
have the global perverse sheaf
\begin{align*}
\phi_{\mM_{0}(m)} \in \Perv(\mM_{0}(m)). 
\end{align*}
Then the perverse sheaf
$\phi_{S^m(X)}$ on $S^m(X)$ is defined by 
\begin{align}\label{def:Sm(X)}
\phi_{S^m(X)} \cneq \pH^1(\dR p_{\ast} 
\phi_{\mM_{0}(m)}) \in \Perv(S^m(X)). 
\end{align}
\begin{lem}\label{lem:delta}
The perverse sheaf $\phi_{S^m(X)}$ is isomorphic to 
$\Delta_{X\ast}L[3]$
for a rank one local system $L$ on $X$, where 
$\Delta_X \colon X \hookrightarrow S^m(X)$ is the 
diagonal embedding. 
\end{lem}
\begin{proof}
The question is local on $X$, 
so we may assume that $X=\mathbb{C}^3$. 
In this case, the stack $\mM_0(m)$ is 
the moduli stack of $m$-dimensional 
representations of $(Q_{(3)}, W_{(3)})$
defined in (\ref{quiver:3}), and in 
this case the BPS sheaf 
is computed in~\cite[Section~5.1]{BeDa}:
\begin{align}\label{id:bpssheaf}
\pH^1(\dR p_{\ast}\phi_{\mM_0(m)})=
\phi_{\tr \overline{W}_{(3)}}(\IC_{M_{Q_{(3)}}(m)}) = \Delta_{\mathbb{C}^3} \IC_{\mathbb{C}^3}.
\end{align}
Here
$M_{Q_{(3)}}(m)$ is the 
coarse moduli space of
$Q_{(3)}$-representations of dimension $m$,  
 the first identity is due to (\ref{id:BPS}) and the second one 
is computed in~\cite[Section~5.1]{BeDa}. 
The above local result shows the lemma. 
\end{proof}
For the case of zero cycles, 
the identity 
map $S^m(X) \to S^m(X)$ is the 
HC map from the coarse moduli 
space of $\mM_0^{\rm{red}}(m)$ 
to the Chow variety of zero cycles. 
Therefore Lemma~\ref{lem:delta} immediately 
shows the following:
\begin{thm}\label{thm:zerocycle}
For a smooth projective CY 3-fold
and $\gamma \in S^m(X)$,  
the invariant 
$\Phi_{\sigma}(\gamma, m)$ is 
zero unless 
$\gamma=\Delta_X(x)$ for 
$x \in X$ and $\Delta_X \colon X \hookrightarrow S^m(X)$
is the diagonal embedding. 
When $\gamma=\Delta_X(x)$, we have
\begin{align*}
\Phi_{\sigma}(\gamma, m)=\chi(\Delta_{X\ast}L[3]|_{\Delta_X(x)})=-1.
\end{align*}
\end{thm}

\subsection{The case of elliptic fibrations}
Here we give an example of the invariant (\ref{def:Phi:sigma})
in the case of a CY 3-fold with an elliptic fibration, 
and see that Conjecture~\ref{conj:main} holds in this case. 

Let $S=\mathbb{P}^2$ and take 
general elements $u \in H^0(S, \oO_S(-4K_S))$,
$v \in H^0(S, \oO_S(-6K_S))$. 
Then as in~\cite[Section~6.4]{Tsurvey}, 
we have a simply connected CY 3-fold 
$X$ with a flat elliptic fibration
\begin{align}\label{pi:XS}
\pi_X \colon X \to S
\end{align}
defined by the equation 
$zy^2=uxz^2+vz^3$
in the projective bundle 
\begin{align*}
\mathbb{P}_S(\oO_S(-2K_S) \oplus \oO_S(-3K_S) \oplus \oO_S) \to S. 
\end{align*}
Here $[x:y:z]$ is the homogeneous coordinate
of the above projective bundle. 
Note that $\pi_X$ admits a section
\begin{align*}
\iota \colon S \to X
\end{align*}
whose image correspond to the fiber point 
$[0:1:0]$. 
By the construction,  
every scheme theoretic fiber $X_s=\pi_X^{-1}(s)$ 
for $s \in S$
is an integral curve, 
which is either a smooth elliptic curve, or 
nodal rational curve with one node, or a cuspidal 
rational curve.  
Let $[F] \in N_1(X)$ be a fiber class of $\pi_X$, and 
set $\beta=d[F]$
for $d\in \mathbb{Z}_{\ge 1}$, and $v=(\beta, m) \in \Gamma_X$. 
Let $k \in \mathbb{Z}_{\ge 1}$ be the greatest
common divisor of $(d, m)$, and 
set $d'=d/k$, $m'=m/k$. 
For $B+i\omega \in A(X)_{\mathbb{C}}$, let 
\begin{align*}
\pi_Y \colon 
Y \to S
\end{align*}
be the $\pi_X$-relative moduli space of 
$(B, \omega)$-stable 
sheaves $E$ on the fibers of (\ref{pi:XS})
such that $v(E)=(d'[F], m')$. 
\begin{lem}\label{lem:isomXY}
We have an isomorphism 
$X \stackrel{\cong}{\to} Y$ over $S$. 
\end{lem}
\begin{proof}
By the result of~\cite{B-M2}, the moduli space $Y$ is a 
smooth projective CY 3-fold. 
Let $J \to S$ be the $\pi_X$-relative moduli 
space of rank one torsion free sheaves $E$ on the fibers of 
$\pi_X$ satisfying $v(E)=([F], m')$. 
We have birational maps over $S$
\begin{align*}
\phi_1 \colon 
Y \dashrightarrow J, \ 
\phi_2 \colon X \dashrightarrow J. 
\end{align*}
Here the birational map $\phi_1$ is given by 
sending a general point $y \in Y$ to 
$\det (E_y)$, where 
$E_y$ is the sheaf on $X_{\pi_Y(y)}$ corresponding to 
$y$, and the determinant is taken in the fiber 
$X_{\pi_Y(y)}$. 
The birational map $\phi_2$
is given by sending a general point $x \in X$ to 
$\oO_{X_{\pi_X(x)}}(x+(m-1) \cdot \iota \circ \pi_X(x))$. 
It follows that we have a birational map 
\begin{align}\label{bir:dash}
\phi_1^{-1} \circ \phi_2 \colon
X \dashrightarrow Y
\end{align}
 over $S$. 
Since both of $X, Y$ are smooth projective CY 3-folds, 
the birational map (\ref{bir:dash}) 
has to be decomposed into flops over $S$. 
However as
$\rho(X)=2$ and 
the Mori cone $\overline{NE}(X)$
of $X$ is spanned by 
the fiber $[F]$ and $[\iota(l)]$ for a line $l \subset S$, 
there is no extremal ray on $\overline{NE}(X)$
corresponding to a flop. 
So the birational map (\ref{bir:dash}) extends to the isomorphism 
$X \stackrel{\cong}{\to} Y$ over $S$. 
\end{proof}
Note that if $\beta=d[F]$ we have the isomorphism
\begin{align*}
\pi_X^{\ast} \colon S^d(S) \stackrel{\cong}{\to} \Chow_X(d[F])
\end{align*}
by sending a zero cycle $Z \subset S$ to 
$\pi_X^{-1}(Z)$. 
We identity $\Chow_X(d[F])$ with $S^d(S)$ by the 
above isomorphism. 
Now we show the main result in this subsection: 
\begin{thm}
In the above situation, 
for
$\sigma \in U(X)$, 
 $\gamma \in \Chow_X(d[F])=S^d(S)$ and 
$m \in \mathbb{Z}$, 
we have
$\Phi_{\sigma}(\gamma, m) \neq 0$
only if $\gamma$ is in the image of 
the diagonal map 
$\Delta_S \colon S \hookrightarrow S^d(S)$. 
If $\gamma=\Delta_S(s)$
for $s \in S$, we have 
\begin{align}\label{Phi:exm}
\Phi_{\sigma}(\gamma, m)=y^{-1}+(2-e(X_s))+y.
\end{align}
Here $e(-)$ is the topological euler number. 
In particular Conjecture~\ref{conj:main} holds in this case. 
\end{thm}
\begin{proof}
By Lemma~\ref{lem:isomXY} and the result of~\cite{B-M2}, 
there is an auto-equivalence
\begin{align*}
\Psi \colon 
D^b(\Coh(X)) \stackrel{\sim}{\to} D^b(\Coh(X))
\end{align*}
sending $\oO_x$ for a closed point $x \in X$
to a stable sheaf $E_x$ with $v(E_x)=(d'[F], m')$. 
Let $\mM_0(k)$ be the moduli stack of zero dimensional 
sheaves on $X$ with length $k$. 
Then as in~\cite[Section~6.4]{Tsurvey}, 
the equivalence $\Psi$ induces the isomorphism of stacks 
\begin{align*}
\Psi_{\ast} \colon \mM_0(k) \stackrel{\cong}{\to}
\mM_{\sigma}(v)
\end{align*}
for any $\sigma \in U(X)$. 
Since the above isomorphism is induced by the 
derived equivalence $\Psi$, it preserves the 
$d$-critical structures and the virtual canonical 
line bundles (see~\cite[Remark~8.6]{MT}, where the 
announced work~\cite[Theorem~1.2]{BrDy} is now available in~\cite[Theorem~5.5]{BrDy2}.). 
Moreover we have the commutative diagram
\begin{align}\label{com:chow}
\xymatrix{
& \mM_0^{\rm{red}}(k) \ar[r]^-{\Psi_{\ast}}_-{\cong} \ar[d]_-{p} & 
\mM_{\sigma}^{\rm{red}}(v) \ar[d]^-{p_{\mM}} \\
X\ar@<-0.3ex>@{^{(}->}[r]^-{\Delta_X} \ar[d]_-{\pi_X} &
S^k(X) \ar[r]^-{\Psi_{\ast}}_-{\cong} 
\ar[d]_-{h} & M_{\sigma}^{\rm{red}}(v) \ar[d] \\
S \ar@<-0.3ex>@{^{(}->}[r]^-{\Delta_S} &
S^d(S) \ar[r]_-{\cong}^-{\pi_X^{\ast}} & \Chow_X(d[F]).
}
\end{align}
Here the 
maps $p, p_{\mM}$ are natural maps to the coarse moduli spaces, the 
middle horizontal arrow is the induced 
isomorphism on the coarse moduli spaces, and
the map $h$ is given by
\begin{align*}
(x_1, \ldots, x_k) \mapsto 
(d' \pi_X(x_1), \ldots, d' \pi_X(x_k)).
\end{align*}

By Proposition~\ref{lem:M0vir}
and the diagram (\ref{com:chow}), 
we have a global orientation data of $\mM_{\sigma}(v)$
satisfying
\begin{align*}
(K_{\mM_{\sigma}(v)}^{\rm{vir}})^{1/2}
 \cong \oO_{\mM_{\sigma}^{\rm{red}}(v)}. 
\end{align*}
Using the above orientation data, we 
have the global perverse sheaf
\begin{align*}
\phi_{\mM_{\sigma}(v)} \in \Perv(\mM_{\sigma}(v)). 
\end{align*}
Then we have the perverse sheaf
on $S^k(X)$ as in (\ref{def:Sm(X)}): 
\begin{align*}
\phi_{S^k(X)} = \pH^1(\dR p_{\ast} 
\Psi_{\ast}^{-1} \phi_{\mM_{\sigma}(v)}) \in \Perv(S^k(X)). 
\end{align*}
By Lemma~\ref{lem:delta} and the simply connectedness of 
$X$, we have 
$\phi_{S^k(X)}=\Delta_{X\ast} \IC_X$. 
By the commutative diagram (\ref{com:chow})
and the definition of $\Phi_{\sigma}(\gamma, m)$, 
we have 
\begin{align*}
\Phi_{\sigma}(\gamma, m)=
\sum_{i\in \mathbb{Z}}
\chi(\pH^i(\dR h_{\ast} \Delta_{X\ast}\IC_X)|_{\gamma}) y^i.
\end{align*}
Therefore $\Phi_{\sigma}(\gamma, m)=0$ if 
$\gamma$ is not in the image of $\Delta_S$. 
If $\gamma=\Delta_S(s)$ for $s\in S$, 
then by the left bottom diagram of (\ref{com:chow}), we have 
\begin{align*}
\Phi_{\sigma}(\gamma, m)=
\sum_{i\in \mathbb{Z}}
\chi(\pH^i(\dR \pi_{X\ast}\IC_X)|_{s}) y^i. 
\end{align*}
Since by our assumption 
each $X_s$ is either a smooth elliptic curve, or 
a rational nodal curve with one node, or a cuspidal rational curve, 
the perverse decomposition of $\dR \pi_{X\ast}\IC_X$ becomes
\begin{align*}
\dR \pi_{X\ast}\IC_X=\IC_S[1] \oplus V \oplus \IC_S[-1].
\end{align*}
Here $V$ is a constructible sheaf 
on $S$ such that for $s \in S$, 
we have $V|_{s}=\mathbb{Q}^{2-e(X_s)}$. 
Therefore the identity (\ref{Phi:exm}) holds. 
\end{proof}

\section{Wall-crossing formula for quivers with convergent super-potentials}\label{sec:wcf}
The wall-crossing formula of cohomological DT invariants for representations of
quivers with super-potentials was studied in~\cite{DaMe}. 
Our approach toward Conjecture~\ref{conj:main}
is to reduce the problem to the similar 
problem for representations of quivers with 
formal but convergent super-potentials, using the result of~\cite{Todstack}.
In this section, we review the work of~\cite{DaMe} and 
prove some necessary results in the case of 
quivers with
convergent super-potentials.
\subsection{Representations of quivers}
Recall that a \textit{quiver} $Q$ consists data
\begin{align*}
Q=(V(Q), E(Q), s, t)
\end{align*}
where $V(Q), E(Q)$ are finite sets 
and $s, t$ are maps
\begin{align*}
s, t \colon E(Q) \to V(Q).
\end{align*}
The set $V(Q)$ is the set of vertices and
$E(Q)$ is the set of edges. 
For $e \in E(Q)$, 
$s(e)$ is the source of $e$
and $t(e)$ is the target of $e$. 
For $i, j \in V(Q)$, we use the following notation
\begin{align}\label{Eab}
E_{i, j} \cneq \{e \in E(Q) : 
s(e)=i, t(e)=j\}
\end{align}
i.e. $E_{i, j}$ is the set of edges 
from $i$ to $j$. 
 
A \textit{$Q$-representation} consists of
data
\begin{align}\label{rep:Q}
\mathbb{V}=\{
(V_i, u_e) : \ i \in V(Q),  \ e \in E(Q), \ 
u_e \colon V_{s(e)} \to V_{t(e)}\}
\end{align}
where $V_i$ is a finite dimensional 
$\mathbb{C}$-vector space 
and $u_e$ is a linear map. 
For a $Q$-representation (\ref{rep:Q}), the vector
\begin{align}\label{m:vect}
\vec{m}=(m_i)_{i \in V(Q)}, \ 
m_i=\dim V_i
\end{align}
is called the \textit{dimension vector}. 

Given a dimension vector (\ref{m:vect}), 
let $V_i$ be a $\mathbb{C}$-vector space with 
dimension $m_i$. 
Let us set 
\begin{align*}
G \cneq \prod_{i \in Q(V)} \GL(V_i), \ 
\mathrm{Rep}_Q(\vec{m}) \cneq \prod_{e \in E(V)} \Hom(V_{s(e)}, V_{t(e)}).
\end{align*}
The algebraic group $G$ acts on $\mathrm{Rep}_Q(\vec{m})$ by 
\begin{align}\label{G:act}
g \cdot u=\{g_{t(e)}^{-1} \circ u_e \circ g_{s(e)}\}_{e\in E(Q)}
\end{align}
for $g=(g_i)_{i \in V(Q)} \in G$ 
and $u=(u_e)_{e\in E(Q)}$. 
A $Q$-representation with dimension vector $\vec{m}$ is 
determined by a point in $\mathrm{Rep}_Q(\vec{m})$
up to $G$-action. 
The moduli stack of $Q$-representations with 
dimension vector $\vec{m}$ is given by the 
quotient stack 
\begin{align*}
\mM_{Q}(\vec{m}) \cneq \left[ \mathrm{Rep}_Q(\vec{m})/G \right]. 
\end{align*}
It has the coarse moduli space, given by
\begin{align}\label{mor:coarse}
M_{Q}(\vec{m}) \cneq \mathrm{Rep}_Q(\vec{m}) \sslash G. 
\end{align}
Here in general, if a reductive algebraic group $G$ acts on 
an affine scheme $Y=\Spec R$, its GIT quotient is given by
$Y\sslash G \cneq \Spec (R^G)$. 
A closed point of $M_{Q}(\vec{m})$ corresponds to a
semi-simple $Q$-representation, i.e. 
direct sum of simple $Q$-representations. 
We have the natural commutative diagram
\begin{align}\label{dia:quiver}
\xymatrix{
\mathrm{Rep}_Q(\vec{m}) \ar[r] \ar[rd]_{\pi_Q} & \mM_{Q}(\vec{m}) \ar[d]^{p_Q} \\
& M_{Q}(\vec{m}). 
}
\end{align}
In what follows, we consider only 
symmetric quivers defined below. 
\begin{defi}\label{def:symmetric}
A quiver $Q$ is called symmetric if 
$\sharp E_{i, j}=\sharp E_{j, i}$ for any 
$i, j \in V(Q)$. 
Here $E_{i, j}$ is defined as in (\ref{Eab}). 
\end{defi}
Let $\mM_Q$ and $M_Q$ be defined by
\begin{align*}
\mM_Q \cneq \coprod_{\vec{m}>0}
\mM_Q(\vec{m}), \ 
M_Q \cneq \coprod_{\vec{m}>0}
M_Q(\vec{m}). 
\end{align*}
Here $\vec{m}>0$ means $m_i \ge 0$ for all $i$ 
and $\vec{m} \neq 0$. 
For each $n\ge 1$, there 
is a natural map
\begin{align}\label{plus}
\oplus \colon \overbrace{M_Q \times \ldots \times M_Q}^{n} \to M_Q
\end{align}
by taking the direct sum of the corresponding 
semi-simple $Q$-representations.
Then we have the map
\begin{align}\label{per:box}
\boxtimes_{\oplus} \colon \Perv(M_Q)^{\times n} 
\to \Perv(M_Q)
\end{align}
by sending $(\fF_1, \ldots, \fF_n)$ to 
\begin{align*}
\boxtimes_{\oplus}(\fF_1, \ldots, \fF_n) 
\cneq \oplus_{\ast}(\fF_1 \boxtimes \cdots \boxtimes \fF_n). 
\end{align*}
Here we note that, since the map (\ref{plus}) is a finite
map, 
we have $\oplus_{\ast}=\dR \oplus_{\ast}$ and it
takes perverse sheaves to perverse sheaves.  
Then the map
\begin{align}\label{per:sym}
\mathrm{Sym}^{\bullet} \colon 
\Perv(M_Q) \to \Perv(M_Q)
\end{align}
is defined by sending 
$\fF$ to 
\begin{align*}
\mathrm{Sym}^{\bullet}(\fF) \cneq 
\bigoplus_{n\ge 1}( \boxtimes_{\otimes}
(\overbrace{\fF, \ldots, \fF}^n))^{S_n}. 
\end{align*}
It is easy to see that 
$\mathrm{Sym}^{\bullet}(\fF)$ is a 
finite sum on each component
$M_Q(\vec{m})$, so it is 
well-defined (see~\cite[Section~3.2]{DaMe}). 
The following result was proved in~\cite{DaMe}.
\begin{thm}\emph{(\cite[Theorem~4.7]{DaMe})}\label{thm:IC1}
For a symmetric quiver $Q$, we 
have an isomorphism of 
perverse sheaves on $M_Q$
\begin{align}
\bigoplus_{i\in \mathbb{Z}}
\pH^i(\dR p_{Q\ast}\IC_{\mM_Q})
[-i] \cong
\Sym^{\bullet}(j_{!\ast}^s \IC_{M_Q^{s}} \otimes 
H^{\ast}(\mathbb{P}^{\infty})_{\rm{vir}}). 
\end{align}
Here $j^s \colon M_Q^s \subset M_Q$ is the open immersion 
of the simple part, and 
$H^{\ast}(\mathbb{P}^{\infty})_{\rm{vir}}$ is defined by
\begin{align*}
H^{\ast}(\mathbb{P}^{\infty})_{\rm{vir}}
\cneq \bigoplus_{k\ge 0} \mathbb{Q}[-2k-1].
\end{align*}
\end{thm}

\subsection{Semistable quiver representations}
For a quiver $Q$, let $K(Q)$ be the Grothendieck group of the 
abelian category of 
finite dimensional 
$Q$-representations. 
For each $i \in V(Q)$ let 
$S_i$ be the one dimensional
 $Q$-representation corresponding to 
the vertex $i$, whose dimension vector is denoted by 
$\mathbf{dim} (S_i)$. By taking the dimension vectors
of $Q$-representations, we have the
group homomorphism
\begin{align}\label{K:dim}
\mathbf{dim} \colon 
K(Q) \to \Gamma_Q \cneq \bigoplus_{i \in V(Q)} \mathbb{Z} \cdot 
\mathbf{dim} (S_i).
\end{align}
Let 
$\hH \subset \mathbb{C}$ be the upper half plane, and take 
\begin{align}\label{xi}
\xi=(\xi_i)_{i \in V(Q)}, \ \xi_i \in \hH.  
\end{align}
Let 
$Z_{\xi}$ be the group homomorphism
\begin{align*}
Z_{\xi} \colon K(Q) \stackrel{\mathbf{dim}}{\to} \Gamma_Q \to
 \mathbb{C}, \ 
[S_i] \mapsto \xi_i. 
\end{align*}
Then $Z_{\xi}$ defines a Bridgeland stability 
condition on the category of finite dimensional $Q$-representations
w.r.t. the group homomorphism (\ref{K:dim}).  
The associated (semi)stable representations
are described in terms of the slope function $\mu_{\xi}$
defined by
\begin{align*}
\mu_{\xi}(-) \cneq -\frac{\Ree Z_{\xi}(-)}{\Imm Z_{\xi}(-)}. 
\end{align*}
\begin{defi}
A $Q$-representation $\mathbb{V}$ is called $\mu_{\xi}$-(semi)stable 
if for any non-zero 
sub $Q$-representation $\mathbb{V}' \subsetneq \mathbb{V}$, we have the inequality
\begin{align*}
\mu_{\xi}(\mathbb{V}')<(\le) \mu_{\xi}(\mathbb{V}).
\end{align*}
\end{defi}
For a choice of $\xi$ as in (\ref{xi}), let 
\begin{align*}
\mathrm{Rep}_Q^{\xi}(\vec{m}) \subset \mathrm{Rep}_Q(\vec{m})
\end{align*}
be the open locus consisting of $\mu_{\xi}$-semistable objects. 
We 
take the associated GIT quotients: 
\begin{align*}
\mM_{Q}^{\xi}(\vec{m}) \cneq [\mathrm{Rep}_Q^{\xi}(\vec{m})/G], \ 
M_{Q}^{\xi}(\vec{m}) \cneq \mathrm{Rep}_Q^{\xi}(\vec{m})\sslash G.
\end{align*}
We have the commutative diagram
\begin{align}\label{com:MQ}
\xymatrix{
\mM_{Q}^{\xi}(\vec{m}) \ar@<-0.3ex>@{^{(}->}[r]^{j_Q^{\xi}} \ar[d]_-{p_Q^{\xi}}
\ar[rd]^-{r_Q^{\xi}} & 
\mM_{Q}(\vec{m}) \ar[d]^-{p_Q} \\
M_Q^{\xi}(\vec{m}) \ar[r]_-{q^{\xi}_Q} & M_Q(\vec{m}).
}
\end{align}
Here $j_Q^{\xi}$ is an open immersion, 
$p_Q$, $p_Q^{\xi}$ are natural morphisms to the 
coarse moduli spaces, and $q_Q^{\xi}$ is the induced 
morphism by the universality of the coarse moduli spaces. 

Recall that a morphism of algebraic varieties 
$f \colon S \to T$ is \textit{semismall} if $f$ is 
proper, surjective, and 
there is a stratification 
$\{S_{\theta}\}_{\theta}$ of $T$ such that 
for any $x \in T$ we have 
\begin{align}\label{cond:semismall}
	\dim f^{-1}(x) \le \frac{1}{2}\codim S_{\theta}.
\end{align}
We call $f$ a \textit{quasi semismall} if 
there is a stratification 
$\{S_{\theta}\}_{\theta}$ of $T$ such that 
the condition (\ref{cond:semismall}) holds (i.e. $f$ is not assumed 
to be proper nor surjective).
We will use the following lemma: 
\begin{lem}\label{lem:small}
For a symmetric quiver $Q$, the 
morphism 
$q^{\xi}_Q \colon M_Q^{\xi}(\vec{m}) \to M_Q(\vec{m})$ is 
a semismall map if $M_Q^{\xi}(\vec{m}) \neq \emptyset$. 
\end{lem}
\begin{proof}
It is well-known that $q_Q^{\xi}$ is projective (in particular proper) and 
surjective (see~\cite[Theorem~4.1]{Hille}). 
It is enough to show that $q_Q^{\xi}$ is quasi semismall. 
For a dimension vector $\vec{w}$ of $Q$, 
let $M_{Q}^{\vec{w}}(\vec{m})$ be the 
moduli space of 
$Q$-representations $\mathbb{V}$ as in (\ref{rep:Q}) with dimension 
vector $\vec{m}$ together with 
linear maps
\begin{align}\label{generate}
W_i \to V_i, \ i \in V(Q)
\end{align}
where $W_i$ is a $w_i$-dimensional vector space, 
such that the image of (\ref{generate}) generates
$\mathbb{V}$ as $\mathbb{C}[Q]$-module, where 
$\mathbb{C}[Q]$ is the path algebra of $Q$ 
(see Subsection~\ref{subsec:conv} below). 
The moduli space $M_{Q}^{\vec{w}}(\vec{m})$ is 
a non-singular variety, and we 
have the forgetting morphism
\begin{align}\notag
\pi_{Q}^{\vec{w}} \colon 
M_{Q}^{\vec{w}}(\vec{m}) \to
\mM_{Q}(\vec{m}) \stackrel{p_Q}{\to} M_Q(\vec{m}).
\end{align} 
Here the first arrow is a smooth morphism of 
relative dimension 
$\vec{w} \cdot \vec{m}$, 
which is surjective when $\vec{w} \gg 0$. 
Moreover by~\cite[Theorem~1.4]{MeRe}, 
the morphism $\pi_{Q}^{\vec{w}}$ satisfies 
the following: 
there is a stratification 
$\{S_{\theta}\}_{\theta}$ of $M_Q(\vec{m})$ such that 
for any $x \in S_{\theta}$ we have 
\begin{align*}
\dim (\pi_{Q}^{\vec{w}})^{-1}(x) 
\le \frac{1}{2}\codim S_{\theta} + 
\vec{w} \cdot \vec{m}-1. 
\end{align*}
Therefore for $x \in S_{\theta}$,
we have
\begin{align}\label{p:small}
\dim (p_{Q})^{-1}(x) 
\le \frac{1}{2}\codim S_{\theta}-1.
\end{align}
Let $\mM_Q^{\xi, s}(\vec{m})$, 
$M_Q^{\xi, s}(\vec{m})$ be the stable locus 
of $\mM_Q^{\xi}(\vec{m})$, $M_Q^{\xi}(\vec{m})$ respectively. 
Then $\mM_Q^{\xi, s}(\vec{m})$ is a $\mathbb{C}^{\ast}$-gerbe
over $M_Q^{\xi, s}(\vec{m})$, so the diagram (\ref{com:MQ})
and (\ref{p:small}) imply 
that 
\begin{align*}
q_{Q}^{\xi}|_{M_Q^{\xi, s}(\vec{m})}
\colon M_Q^{\xi, s}(\vec{m})
\to M_Q(\vec{m})
\end{align*}
is a quasi semismall map. 
For dimension vectors 
$\vec{m}_1, \ldots, \vec{m}_s$
whose sum equals to $\vec{m}$
and $\mu_{\xi}(\vec{m}_i)=\mu_{\xi}(\vec{m})$, 
we have the commutative diagram
\begin{align*}
\xymatrix{
M_Q^{\xi, s}(\vec{m}_1) \times \cdots 
\times M_Q^{\xi, s}(\vec{m}_s) 
\ar[d]_{(q_Q^{\xi}, \ldots, q_Q^{\xi})} \ar[r]^-{\oplus} & 
M_Q^{\xi}(\vec{m}) \ar[d]^{q_{Q}^{\xi}} \\
M_Q(\vec{m}_1) \times \cdots \times 
M_Q(\vec{m}_s) \ar[r]_-{\oplus} & M_Q(\vec{m}).
}
\end{align*}
Since the images of the top horizontal arrows
for various $(\vec{m}_1, \ldots, \vec{m}_s)$
give a stratification of $M_Q^{\xi}(\vec{m})$, and the 
left vertical arrow is a quasi semismall map, 
we conclude that $q_{Q}^{\xi}$ is quasi semismall. 
\end{proof}

\subsection{Wall-crossing formula for IC sheaves}
We keep the notation in the previous subsection. 
For each $\mu \in (-\infty, \infty)$, 
let 
$\mM_Q^{\xi}(\mu), M_Q^{\xi}(\mu)$ be defined by
\begin{align*}
\mM_Q^{\xi}(\mu)
\cneq 
\coprod_{\mu_{\xi}(\vec{m})=\mu}
\mM_Q^{\xi}(\vec{m}), \ 
M_Q^{\xi}(\mu)
\cneq 
\coprod_{\mu_{\xi}(\vec{m})=\mu}
M_Q^{\xi}(\vec{m}). 
\end{align*}
Similarly to $M_Q$, 
a closed point of $M_Q^{\xi}(\mu)$ 
corresponds to a $\mu_{\xi}$-polystable 
$Q$-representation, i.e. 
direct sum of $\mu_{\xi}$-stable $Q$-representations with 
slope $\mu$. 
Therefore we have the natural maps
\begin{align*}
&\boxtimes_{\oplus} \colon 
\Perv(M_Q^{\xi}(\mu))^{\times n} \to 
\Perv(M_Q^{\xi}(\mu)), \\
&\Sym^{\bullet} \colon 
\Perv(M_Q^{\xi}(\mu)) \to \Perv(M_Q^{\xi}(\mu))
\end{align*}
defined similarly to (\ref{per:box}), (\ref{per:sym}). 
We refer to the following results 
proved in~\cite{DaMe}.
\begin{thm}\emph{(\cite[Theorem~4.11]{DaMe})}\label{thm:IC2}
For a symmetric quiver $Q$, 
we have the following: 

\begin{enumerate}
\item We have an isomorphism of 
perverse sheaves on $M_Q^{\xi}(\mu)$
\begin{align}
\bigoplus_{i\in \mathbb{Z}}
\pH^i(\dR p_{Q\ast}^{\xi}\IC_{\mM_Q^{\xi}(\mu)})
[-i] \cong
\Sym^{\bullet}(j_{!\ast}^{\xi, s} \IC_{M_Q^{\xi, s}(\mu)} \otimes 
H^{\ast}(\mathbb{P}^{\infty})_{\rm{vir}}). 
\end{align}
Here $j^{\xi, s} \colon M_Q^{\xi, s}(\mu) \subset M_Q^{\xi}(\mu)$
is the open immersion of the stable part. 
\item We have an isomorphism of perverse sheaves on $M_Q$
\begin{align*}
&\bigoplus_{i\in \mathbb{Z}}
\pH^i(\dR p_{Q\ast}\IC_{\mM_Q})
[-i] \\
&\cong
\scalebox{1.5}{$\boxtimes$}_{\oplus, \infty \stackrel{\mu}{\to} -\infty}
\dR q_{Q \ast}^{\xi}\Sym^{\bullet}(j_{!\ast}^{\xi, s} \IC_{M_Q^{\xi, s}(\mu)} 
\otimes H^{\ast}(\mathbb{P}^{\infty})_{\rm{vir}}). 
\end{align*}
\end{enumerate}
\end{thm}
We will use the following lemma: 
\begin{lem}\label{lem:perv}
In the diagram (\ref{com:MQ}), 
for any $i\in \mathbb{Z}$ we have
\begin{align}\label{perv:M}
\dR q_{Q \ast}^{\xi} \pH^i(\dR p_{Q\ast}^{\xi} \IC_{\mM_Q^{\xi}(\vec{m})})
\in \Perv(M_Q(\vec{m})). 
\end{align}
\end{lem}
\begin{proof}

By Lemma~\ref{lem:small}
and a general fact that 
the derived push-forward of 
semismall maps take intersection complexes to 
perverse sheaves (for example see~\cite{ACLM}), 
we have
\begin{align}\label{IC:P}
\dR q_{Q \ast}^{\xi}
j_{!\ast}^{\xi, s} \IC_{M_Q^{\xi, s}(\mu)}
\in \Perv(M_Q(\vec{m})). 
\end{align}
Then the lemma follows from Theorem~\ref{thm:IC2} (i), 
the condition (\ref{IC:P}) and 
the commutative diagram
\begin{align*}
\xymatrix{
M_Q^{\xi}(\mu) \times \cdots \times M_Q^{\xi}(\mu)
\ar[d]_{(q_Q^{\xi}, \ldots, q_Q^{\xi})} \ar[r]^-{\oplus} & 
M_Q^{\xi}(\mu) \ar[d]^{q_{Q}^{\xi}} \\
M_Q \times \cdots \times 
M_Q \ar[r]_-{\oplus} & M_Q. 
}
\end{align*}
\end{proof}
In the diagram (\ref{com:MQ}), 
we have the canonical morphism
\begin{align*}
\IC_{\mM_Q(\vec{m})} \to \dR j_{Q\ast}^{\xi}
j_Q^{\xi\ast}\IC_{\mM_Q(\vec{m})}
=\dR j_{Q\ast}^{\xi}
\IC_{\mM_Q^{\xi}(\vec{m})}
\end{align*}
by adjunction. 
By pushing forward it to $M_Q(\vec{m})$, we have
the morphism
\begin{align*}
\dR p_{Q\ast}\IC_{\mM_Q(\vec{m})}
\to \dR p_{Q\ast}\dR j_{Q\ast}^{\xi}\IC_{\mM_Q^{\xi}(\vec{m})}
=\dR q_{Q \ast}^{\xi} \dR p_{Q\ast}^{\xi}\IC_{\mM_Q^{\xi}(\vec{m})}. 
\end{align*}
By Lemma~\ref{lem:perv}, taking the first perverse 
cohomology of the above morphism gives
the morphism of perverse sheaves on $M_Q(\vec{m})$: 
\begin{align}\label{per:j}
\pH^1(\dR p_{Q\ast}\IC_{\mM_Q(\vec{m})})
\to \dR q_{Q \ast}^{\xi} 
\pH^1(\dR p_{Q\ast}^{\xi}\IC_{\mM_Q^{\xi}(\vec{m})}). 
\end{align}
\begin{lem}\label{isom:perj}
The morphism (\ref{per:j}) is an isomorphism. 
\end{lem}
\begin{proof}
By taking the first perverse 
cohomologies of the isomorphisms in 
Theorem~\ref{thm:IC1}
and Theorem~\ref{thm:IC2} (i),  
we have
\begin{align*}
&\pH^1(\dR p_{Q\ast}\IC_{\mM_Q(\vec{m})})
\cong j_{!\ast}^s\IC_{M_Q^s(\vec{m})}, \\ 
&\pH^1(\dR p_{Q\ast}^{\xi}\IC_{\mM_Q^{\xi}(\vec{m})})
\cong 
j_{!\ast}^{\xi, s} \IC_{M_Q^{\xi, s}(\vec{m})}.
\end{align*}
Moreover by Theorem~\ref{thm:IC2} (ii) and (\ref{IC:P}), 
we have an isomorphism
\begin{align*}
\pH^1(\dR p_{Q\ast}\IC_{\mM_Q(\vec{m})})
\cong \dR q_{Q \ast}^{\xi} 
j_{!\ast}^{\xi, s} \IC_{M_Q^{\xi, s}(\vec{m})}.
\end{align*}
Therefore  if $M_Q^s(\vec{m}) = \emptyset$, 
then both sides of (\ref{per:j}) are zero. 
Otherwise 
the morphism (\ref{per:j}) is a non-zero
endomorphism of the simple 
perverse sheaf
$j_{!\ast}^{s} \IC_{M_Q^{s}(\vec{m})}$, 
so it is an isomorphism. 
\end{proof}
\subsection{Quivers with convergent super-potentials}\label{subsec:conv}
Recall that a \textit{path}
of a quiver $Q$ 
is a composition of edges in $Q$
\begin{align*}
e_1 e_2 \ldots e_n, \ e_i \in E(Q), \ t(e_i)=s(e_{i+1}). 
\end{align*}
The number $n$ above is called the \textit{length} of the path. 
The \textit{path algebra} of 
a quiver $Q$ is
a $\mathbb{C}$-vector space spanned by 
paths in $Q$:
\begin{align*}
\mathbb{C}[Q] \cneq 
\bigoplus_{n\ge 0}
\bigoplus_{e_1, \ldots, e_n \in E(Q), t(e_i)=s(e_{i+1})} \mathbb{C} \cdot e_1 e_2 \ldots e_n.
\end{align*}
Here a path of length zero is a trivial path 
at each vertex of $Q$, and 
the product on $\mathbb{C}[Q]$ is defined by the 
composition of paths. 
By taking the completion of $\mathbb{C}[Q]$ with respect to the 
length of the path, 
we obtain the formal 
path algebra:
\begin{align*}
\mathbb{C}\lkakko Q \rkakko \cneq 
\prod_{n\ge 0}
\bigoplus_{e_1, \ldots, e_n \in E(Q), t(e_i)=s(e_{i+1})} \mathbb{C} \cdot e_1 e_2 \ldots e_n.
\end{align*}
Note that an element $f \in \mathbb{C}\lkakko Q \rkakko$
is written as 
\begin{align}\label{f:element}
f=\sum_{n\ge 0, \{1, \ldots, n+1\} \stackrel{\psi}{\to} V(Q)}
\sum_{e_i \in E_{\psi(i), \psi(i+1)}}
a_{\psi, e_{\bullet}} \cdot e_1 e_2\ldots e_{n}. 
\end{align}
Here 
$a_{\psi, e_{\bullet}} \in \mathbb{C}$, 
$e_{\bullet}=(e_1, \ldots, e_n)$ and 
$E_{\psi(i), \psi(i+1)}$ is defined as in (\ref{Eab}). 
The above element $f$ lies in $\mathbb{C}[Q]$ iff
$a_{\psi, e_{\bullet}}=0$ for $n\gg 0$. 

The subalgebra
\begin{align*}
\mathbb{C}\{ Q\} \subset \mathbb{C}\lkakko Q \rkakko
\end{align*}
is defined 
to be elements (\ref{f:element}) 
such that $\lvert a_{\psi, e_{\bullet}} \rvert <C^n$ for 
some constant $C>0$ which is independent of $n$. 
Note that $\mathbb{C}\{Q\}$ contains $\mathbb{C}[Q]$ as 
a subalgebra. 
A \textit{convergent super-potential} of a quiver $Q$ is an element 
\begin{align*}
W \in \mathbb{C}\{ Q \}/[\mathbb{C}\{ Q \}, \mathbb{C}\{ Q \}]. 
\end{align*}
A convergent super-potential $W$ of $Q$ is represented by 
a formal sum
\begin{align}\notag
W=\sum_{n\ge 1}
\sum_{\begin{subarray}{c}
\{1, \ldots, n+1\} \stackrel{\psi}{\to} V(Q), \\
\psi(n+1)=\psi(1)
\end{subarray}}
\sum_{e_i \in E_{\psi(i), \psi(i+1)}}
a_{\psi, e_{\bullet}} \cdot e_1 e_2\ldots e_{n}
\end{align}
with $\lvert a_{\psi, e_{\bullet}} \rvert <C^n$
for a constant $C>0$. 

For a dimension vector $\vec{m}$
of $Q$, let $\tr W$ be the formal function 
of $u=(u_e)_{e\in E(Q)} \in \mathrm{Rep}_Q(\vec{m})$ defined by
\begin{align*}
\tr W(u) \cneq \sum_{n\ge 1}
\sum_{\begin{subarray}{c}
\{1, \ldots, n+1\} \stackrel{\psi}{\to} V(Q) \\
\psi(n+1)=\psi(1)
\end{subarray}}
\sum_{e_i \in E_{\psi(i), \psi(i+1)}}
a_{\psi, e_{\bullet}} \cdot \tr(u_n \circ u_{n-1} 
\circ \cdots \circ u_1).
\end{align*}
The above formal function on 
$\mathrm{Rep}_Q(\vec{m})$ is $G$-invariant. 
By the argument of~\cite[Lemma~2.10]{Todstack}, 
there is an 
analytic open neighborhood 
\begin{align}\label{V:open}
0 \in V \subset M_Q(\vec{m})
\end{align}
such that the formal function 
$\tr W$
absolutely converges on 
$\pi_Q^{-1}(V)$ to give a 
$G$-invariant analytic function
\begin{align}\label{tr:W}
\tr W \colon \pi_Q^{-1}(V) \to \mathbb{C}. 
\end{align}
Here $\pi_Q$ is given in the diagram (\ref{dia:quiver}). 
Then we set
\begin{align}\label{def:MW}
\mathrm{Rep}_{(Q, \partial W)}(\vec{m})|_{V} &\cneq 
\{ d(\tr W)=0\}, \\ 
\notag
\mM_{(Q, \partial W)}(\vec{m})|_{V}& \cneq \left[\{ d(\tr W)=0\}/G \right], \\
\notag
M_{(Q, \partial W)}(\vec{m})|_{V}& \cneq \{ d(\tr W)=0\} \sslash G. 
\end{align}
Here $(-)\sslash G$ above is an analytic Hilbert quotient 
(see~\cite{MR1631577, MR3394374, Todstack}). 
We have the natural commutative diagram
\begin{align}\label{diagram:MQ}
\xymatrix{
\mathrm{Rep}_{(Q, \partial W)}(\vec{m})|_{V} \ar[d] \ar@/_40pt/[dd]_-{\pi_{(Q, \partial W)}}
\ar@<-0.3ex>@{^{(}->}[r] & \pi_Q^{-1}(V) \ar[d] \ar@<-0.3ex>@{^{(}->}[r] 
\ar@{}[dr]|\square
& \mathrm{Rep}_Q(\vec{m}) \ar[d] \ar@/^40pt/[dd]^-{\pi_Q} \\
\mM_{(Q, \partial W)}(\vec{m})|_{V} \ar[d]^{p_{(Q, \partial W)}}
\ar@<-0.3ex>@{^{(}->}[r] & p_Q^{-1}(V) \ar[d] \ar@<-0.3ex>@{^{(}->}[r] 
\ar@{}[dr]|\square
& \mM_Q(\vec{m}) \ar[d]^{p_Q} \\
M_{(Q, \partial W)}(\vec{m})|_{V}
\ar@<-0.3ex>@{^{(}->}[r] & V  \ar@<-0.3ex>@{^{(}->}[r] 
& M_Q(\vec{m}).
}
\end{align}
Here the right horizontal arrows are open immersions
and the left horizontal arrows are closed immersions.

Let $\xi$ be as in (\ref{xi}) which defines the 
$\mu_{\xi}$-stability on the category of $Q$-representations, and 
\begin{align}\label{rep:xi}
\mathrm{Rep}_{(Q, \partial W)}^{\xi}(\vec{m})|_{V}
\subset \mathrm{Rep}_{(Q, \partial W)}(\vec{m})|_{V}
\end{align}
be the open locus consisting of $\mu_{\xi}$-semistable 
$Q$-representations. Similarly to (\ref{def:MW}), we define
\begin{align*}
&\mM_{(Q, \partial W)}^{\xi}(\vec{m})|_{V}
\cneq [\mathrm{Rep}_{(Q, \partial W)}^{\xi}(\vec{m})|_{V}/G], \\ 
&M_{(Q, \partial W)}^{\xi}(\vec{m})|_{V}
\cneq \mathrm{Rep}_{(Q, \partial W)}^{\xi}(\vec{m})|_{V}\sslash G. 
\end{align*}
Then we have the commutative diagram
\begin{align}\label{dia:arrow}
\xymatrix{
 \mM_{(Q, \partial W)}^{\xi}(\vec{m})|_{V} 
\ar@<-0.3ex>@{^{(}->}[r]^{j_{(Q, \partial W)}^{\xi}}
\ar[d]_-{p_{(Q, \partial W)}^{\xi}} \ar[rrd]
& \mM_{(Q, \partial W)}(\vec{m})|_{V} \ar[rrd] \ar[d]^-{p_{(Q, \partial W)}}|\hole  & & \\
M_{(Q, \partial W)}^{\xi}(\vec{m})|_{V} \ar[r]^-{q_{(Q, \partial W)}^{\xi}} \ar[rrd] & M_{(Q, \partial W)}(\vec{m})|_{V} \ar[rrd]|(.55)\hole
& \mM_{Q}^{\xi}(\vec{m}) \ar@<-0.3ex>@{^{(}->}[r]_{j_Q^{\xi}} \ar[d]_-{p_Q^{\xi}} & \mM_{Q}(\vec{m}) \ar[d]^-{p_Q} \\
& & M_Q^{\xi}(\vec{m}) \ar[r]_-{q^{\xi}_Q} & M_Q(\vec{m}).
}
\end{align}
Here $j_{(Q, \partial W)}^{\xi}$ is an open immersion, the morphisms 
$p_{(Q, \partial W)}^{\xi}$, $p_{(Q, \partial W)}$ are the natural morphisms to the 
coarse moduli spaces, 
$q_{(Q, \partial W)}^{\xi}$ is the induced morphism by the universality 
of analytic Hilbert quotients
and the slanting 
arrows are locally closed embeddings. 

\subsection{Vanishing cycles for quivers with convergent super-potentials}
We have the following proposition on the analytic stack 
$\mM_{(Q, \partial W)}(\vec{m})|_{V}$
given in (\ref{def:MW}):
\begin{prop}\label{prop:vir:tri}
The analytic stack $\mM_{(Q, \partial W)}(\vec{m})|_{V}$
has an analytic $d$-critical structure 
given by 
\begin{align*}
d(\tr W) +(d(\tr W))^2 \in H^0(\sS_{\mM_{(Q, \partial W)}(\vec{m})|_{V}}^{0}).
\end{align*}
Here $\tr W$ is the function (\ref{tr:W}). 
Moreover if $Q$ is symmetric, there is an 
orientation data of 
$\mM_{(Q, \partial W)}(\vec{m})|_{V}$ 
which is trivial as a line bundle. 
\end{prop}
\begin{proof}
The first statement follows 
from the definition of $\mM_{(Q, \partial W)}(\vec{m})|_{V}$
in (\ref{def:MW})
and Example~\ref{exam:dstack}. 
Below we prove the second statement. 
The virtual canonical line bundle
of $\mM_{(Q, \partial W)}(\vec{m})|_{V}$, 
regarded as 
a $G$-equivariant line bundle on 
$\mathrm{Rep}_{(Q, \partial W)}^{\rm{red}}(\vec{m})|_{V}$, 
 is given by
\begin{align*}
K_{\mM_{(Q, \partial W)}(\vec{m})|_{V}}^{\rm{vir}}
=K_{\mathrm{Rep}_Q(\vec{m})}^{\otimes 2}
|_{\mathrm{Rep}_{(Q, \partial W)}^{\rm{red}}(\vec{m})|_{V}}
\in \Pic_G(\mathrm{Rep}_{(Q, \partial W)}^{\rm{red}}(\vec{m})|_{V}). 
\end{align*}
Here $\Pic_G(-)$ is the group of 
$G$-equivariant line bundles of $(-)$. 
Therefore we have an orientation 
data of $\mM_{(Q, \partial W)}(\vec{m})|_{V}$ given by
\begin{align}\label{orient:quiver}
K_{\mM_{(Q, \partial W)}(\vec{m})|_{V}}^{\rm{vir}, 1/2}
=K_{\mathrm{Rep}_Q(\vec{m})}
|_{\mathrm{Rep}_{(Q, \partial W)}^{\rm{red}}(\vec{m})|_{V}}. 
\end{align}
So the second statement follows if 
for a symmetric quiver $Q$ we have 
\begin{align}\label{Kvir:Q}
K_{\mathrm{Rep}_Q(\vec{m})}\cong 
\oO_{\mathrm{Rep}_{Q}(\vec{m})}
\end{align}
as $G$-equivariant line bundles on 
$\mathrm{Rep}_{Q}(\vec{m})$. 
Since $\mathrm{Rep}_Q(\vec{m})$ is an affine space, 
we have an isomorphism (\ref{Kvir:Q}) as line bundles. 
The $G$-equivariant structure on 
$K_{\mathrm{Rep}_Q(\vec{m})}$ 
is given by the following character 
$G \to \mathbb{C}^{\ast}$
\begin{align*}
(g_i)_{i \in V(Q)} \mapsto 
\prod_{e \in E(Q)} (\det g_{s(e)})^{m_{t(e)}}
 \cdot (\det g_{t(e)})^{-m_{s(e)}}. 
\end{align*}
The symmetric condition of $Q$, 
$\sharp E_{i, j}=\sharp E_{j, i}$ implies that the 
above character is 
trivial $G \to 1 \in \mathbb{C}^{\ast}$. 
Therefore we have an isomorphism 
(\ref{Kvir:Q}) as $G$-equivariant line bundles. 
\end{proof}

Using the orientation data (\ref{orient:quiver})
of $\mM_{(Q, \partial W)}(\vec{m})|_{V}$,
which is trivial as a line bundle,  
we have the 
perverse sheaf of vanishing cycles
\begin{align}\label{phi:mM}
\phi_{\mM_{(Q, \partial W)}(\vec{m})|_{V}} \in 
\Perv(\mM_{(Q, \partial W)}(\vec{m})|_{V}). 
\end{align}
The above perverse sheaf, regarded 
as a $G$-equivariant perverse sheaf on $\mathrm{Rep}_{(Q, \partial{W}), V}(\vec{m})$, 
is  
nothing but the perverse sheaf of vanishing cycles
of the $G$-invariant function (\ref{tr:W}). 
The above oriented 
$d$-critical stack structure 
on $\mM_{(Q, \partial W)}(\vec{m})|_{V}$
induces the one on its open 
substack $\mM_{(Q, \partial W)}^{\xi}(\vec{m})|_{V}$. 
The associated perverse sheaf of vanishing cycles
\begin{align*}
\phi_{\mM_{(Q, \partial W)}^{\xi}(\vec{m})|_{V}}
\in \Perv(\mM_{(Q, \partial W)}^{\xi}(\vec{m})|_{V})
\end{align*}
is a pull-back of (\ref{phi:mM}) 
by $j_{(Q, \partial W)}^{\xi}$ in the diagram (\ref{dia:arrow}). 
\begin{lem}\label{lem:phivanish}
For any $i\in \mathbb{Z}$, 
in the diagram (\ref{dia:arrow})
we have 
\begin{align*}
\dR q_{(Q, \partial W)\ast}^{\xi}\pH^i(\dR p_{(Q, \partial W)\ast}^{\xi}\phi_{\mM_{(Q, \partial W)}^{\xi}(\vec{m})|_{V}} ) 
\in \Perv(M_{(Q, \partial W)}(\vec{m})|_{V}). 
\end{align*}
\end{lem}
\begin{proof}
Since we have 
$p_{Q\ast}\oO_{\mM_Q(\vec{m})}=\oO_{M_Q(\vec{m})}$, which 
also holds after analytification (see~\cite[Lemma~2.7]{Todstack}), 
the function (\ref{tr:W}) descends to the 
analytic function $\tr \overline{W}$ on $V$
\begin{align*}
\tr W \colon \pi_Q^{-1}(V) \to V \stackrel{\tr \overline{W}} \to \mathbb{C}. 
\end{align*}
We restrict the perverse sheaf (\ref{perv:M})
to $V$ and apply the vanishing cycle functor 
of $\tr \overline{W}$. 
Since the vanishing cycle functor preserves the perverse
t-structures, for any $i \in \mathbb{Z}$ we have 
\begin{align*}
\phi_{\tr \overline{W}}((\dR q_{Q \ast}^{\xi} \pH^i(\dR p_{Q\ast}^{\xi} \IC_{\mM_Q^{\xi}(\vec{m})}))|_{V}) \in \Perv(V). 
\end{align*}
On the other hand, by pulling back the diagram (\ref{com:MQ}) to $V$, 
we obtain the diagram
\begin{align}\label{com:MQ2}
\xymatrix{
(r_Q^{\xi})^{-1}(V)
\ar@<-0.3ex>@{^{(}->}[r]^-{j_Q^{\xi}} \ar[d]_-{p_Q^{\xi}}
\ar[rd]^-{r_Q^{\xi}} & 
p_Q^{-1}(V) \ar[d]^-{p_Q} \ar@/^20pt/[rdd]^{\tr W} & \\
(q_Q^{\xi})^{-1}(V) \ar[r]_-{q^{\xi}_Q} \ar@/_20pt/[rrd]^{\tr^{\xi} 
\overline{W}} & V 
\ar[rd]^-{\tr \overline{W}} & \\
  & & \mathbb{C}. 
}
\end{align}
We have isomorphisms
\begin{align*}
\phi_{\tr \overline{W}}((\dR q_{Q \ast}^{\xi} \pH^i(\dR p_{Q\ast}^{\xi} \IC_{\mM_Q^{\xi}(\vec{m})}))|_{V}) 
&\cong \dR q_{Q \ast}^{\xi}(\phi_{\tr^{\xi} \overline{W}}
\pH^i(\dR p_{Q\ast}^{\xi} \IC_{(r_Q^{\xi})^{-1}(V)})) \\
& \cong \dR q_{Q \ast}^{\xi} \pH^i(\phi_{\tr^{\xi} \overline{W}}
\dR p_{Q\ast}^{\xi} \IC_{(r_Q^{\xi})^{-1}(V)}) \\
&\cong \dR q_{Q \ast}^{\xi}
\pH^i(\dR p_{Q\ast}^{\xi}\phi_{\tr W}(\IC_{(r_Q^{\xi})^{-1}(V)})). 
\end{align*}
Here the first isomorphism 
follows from the compatibility of
vanishing cycle functors with proper push-forward 
(see~\cite[Proposition~4.2.11]{Dimbook}). 
The second isomorphism follows from 
that the vanishing cycle functor preserves the 
perverse t-structure. 
The third isomorphism is proved in~\cite[Proposition~4.3]{DaMe}
when $W$ is a usual super-potential 
$W \in \mathbb{C}[Q]/[\mathbb{C}[Q], \mathbb{C}[Q]]$, and 
the same argument applies for the 
convergent super-potential case. 
Therefore the lemma holds. 
\end{proof}
In the diagram (\ref{dia:arrow}), we
have the canonical morphism
\begin{align*}
\phi_{\mM_{(Q, \partial W)}(\vec{m})|_{V}}
\to
\dR j_{(Q, \partial W)\ast}^{\xi} \phi_{\mM_{(Q, \partial W)}^{\xi}(\vec{m})|_{V}}. 
\end{align*}
By pushing forward it to $M_{(Q, \partial W)}(\vec{m})|_{V}$, 
we obtain the morphism
\begin{align*}
\dR p_{(Q, \partial W)\ast}
\phi_{\mM_{(Q, \partial W)}(\vec{m})|_{V}}
\to \dR q_{(Q, \partial W)\ast}^{\xi}\dR p_{(Q, \partial W)\ast}^{\xi}
\phi_{\mM_{(Q, \partial W)}^{\xi}(\vec{m})|_{V}}. 
\end{align*}
By Lemma~\ref{lem:phivanish}, 
taking the first perverse cohomologies above gives
the morphism
\begin{align}\label{cmor:phi}
\pH^1(
\dR p_{(Q, \partial W)\ast}
&\phi_{\mM_{(Q, \partial W)}(\vec{m})|_{V}}) \\ \notag
&\to \dR q_{(Q, \partial W)\ast}^{\xi} \pH^1(\dR p_{(Q, \partial W)\ast}^{\xi}
\phi_{\mM_{(Q, \partial W)}^{\xi}(\vec{m})|_{V}}).
\end{align}
\begin{lem}\label{lem:cmor:phi}
The morphism (\ref{cmor:phi}) is an isomorphism. 
\end{lem}
\begin{proof}
We use the notation in the proof of Lemma~\ref{lem:phivanish}. 
By the functoriality of vanishing cycle functors, we have 
the following natural commutative diagram
\begin{align*}
	\xymatrix{
\phi_{\tr \overline{W}}(\dR p_{Q\ast}\IC_{\mM_Q(\vec{m})}|_{V})	
\ar[r] \ar[d] & \dR p_{Q\ast}\phi_{\tr W}(\IC_{p_Q^{-1}(V)}) \ar[d] \\
\phi_{\tr \overline{W}}(\dR p_{Q\ast}j_{Q\ast}^{\xi}\IC_{\mM_Q^{\xi}(\vec{m})}|_{V}) \ar[r] & 
\dR p_{Q\ast}j_{Q\ast}^{\xi}\phi_{\tr W}(\IC_{(r_Q^{\xi})^{-1}(V)}).
}
	\end{align*}
We take the first perverse cohomologies of the 
above diagram. 
Noting that the vanishing cycle functor commutes 
with perverse cohomology functors, 
the natural 
isomorphism $\dR p_{Q\ast} j_{Q\ast}^{\xi} \cong \dR q_{Q\ast}^{\xi}\dR p_{Q\ast}^{\xi}$ by the diagram (\ref{com:MQ2}), 
and using Lemma~\ref{lem:phivanish}, 
the above diagram gives the following commutative diagram 
\begin{align*}
	\xymatrix{
		\phi_{\tr \overline{W}}(\pH^1(\dR p_{Q\ast}\IC_{\mM_Q(\vec{m})})|_{V})
		\ar[r] \ar[d] &
		\pH^1(\dR p_{Q\ast} \phi_{\tr W}(\IC_{p_Q^{-1}(V)}))
		\ar[d] \\
		\phi_{\tr \overline{W}}((\dR q_{Q\ast}^{\xi} 
		\pH^1(\dR p_{Q\ast}^{\xi}\IC_{\mM_Q^{\xi}(\vec{m})}))|_{V}) 
		\ar[r] & 
		\dR q_{Q\ast}^{\xi} 
		\pH^1(\dR p_{Q\ast}^{\xi} \phi_{\tr W}(\IC_{(r_Q^{\xi})^{-1}(V)})).
		}
\end{align*}
The left vertical arrow is an isomorphism by Lemma~\ref{isom:perj}. 
Moreover the horizontal arrows can be shown 
to be isomorphisms by the same argument of 
Lemma~\ref{lem:phivanish}. Therefore the right vertical 
arrow is also an isomorphism, which implies that (\ref{cmor:phi})
 is an isomorphism. 
\end{proof}

\section{Wall-crossing formula for GV type invariants}\label{sec:wcf:gv}
In this section, using the results in the previous 
sections and the results in~\cite{Todstack}, we prove Theorem~\ref{intro:thm1}
and Theorem~\ref{intro:thm2}. 
\subsection{Ext-quiver}
Let $X$ be a smooth projective CY 3-fold. 
For $\sigma=\sigma_{B, \omega} \in U(X)$
and $v=(\beta, m) \in \Gamma_X$, 
we consider the stack 
$\mM_{\sigma}(v)$ and its coarse moduli space
$M_{\sigma}(v)$
given 
in Subsection~\ref{subsec:twist}.
We have the natural
morphism
\begin{align*}
p_M \colon \mM_{\sigma}(v) \to M_{\sigma}(v).
\end{align*} 
For a closed 
point $p\in M_{\sigma}(v)$, it 
is represented by a $(B, \omega)$-polystable 
sheaf $E$ of the form
\begin{align}\label{poly}
E=\bigoplus_{i=1}^k V_i \otimes E_i
\end{align}
where $E_i \in \Coh_{\le 1}(X)$ is $(B, \omega)$-stable 
with $\mu_{B, \omega}(E_i)=\mu_{B, \omega}(E)$, 
and $E_i \not\cong E_j$ for $i\neq j$. 

For each $1\le i, j \le k$, 
we fix a finite subset
\begin{align}
E_{i, j} \subset \Ext^1(E_i, E_j)^{\vee}
\end{align}
giving a basis of $\Ext^1(E_i, E_j)^{\vee}$. 
The \textit{Ext-quiver} $Q_{E_{\bullet}}$ 
of $E_{\bullet}$
is defined as follows. 
The set of vertices and edges are given by 
\begin{align*}
V(Q_{E_{\bullet}})=\{1, 2, \ldots, k\}, \ 
E(Q_{E_{\bullet}})=\coprod_{1\le i, j \le k}
E_{i,j}. 
\end{align*}
The maps $s, t \colon 
E(Q_{E_{\bullet}}) \to V(Q_{E_{\bullet}})$
are given by 
\begin{align*}
s|_{E_{i, j}}=i, \ t|_{E_{i, j}}=j. 
\end{align*}
\begin{lem}
The quiver $Q_{E_{\bullet}}$ is symmetric. 
\end{lem}
\begin{proof}
For $1\le i, j \le k$ with $i\neq j$, 
we have $\Hom(E_i, E_j)=0$. 
Therefore we have 
\begin{align*}
\dim \Ext^1(E_j, E_i)-\dim \Ext^1(E_i, E_j) &=
 \sum_{k \in \mathbb{Z}}
(-1)^k \dim \Ext^k(E_i, E_j) \\
&=0. 
\end{align*}
Here the first equality follows from the Serre duality and 
the second equality follows from the Riemann-Roch theorem. 
Therefore the lemma holds. 
\end{proof}

\subsection{Moduli spaces of semistable sheaves via Ext-quivers}
Let us take another stability condition 
\begin{align}\label{sigma+}
\sigma^{+}=\sigma_{B^{+}, \omega^{+}}=
(Z_{B^{+}, \omega^{+}}, \Coh_{\le 1}(X)) \in U(X).
\end{align} 
We
take $\sigma^{+}$
sufficiently close to $\sigma$. 
Then by wall-chamber structure on the space of 
stability conditions, any $\sigma^{+}$-semistable object $E$
with $\ch(E)=v$ is $\sigma$-semistable. 
Then
we have the commutative diagram
\begin{align}\label{dia:M+}
\xymatrix{
\mM_{\sigma^{+}}(v) \ar@<-0.3ex>@{^{(}->}[r] \ar[dr]_{r_M}
 \ar[d]_{p_M^{+}}
 & \mM_{\sigma}(v) 
\ar[d]^{p_M} \\
M_{\sigma^{+}}(v) \ar[r]_{q_M}  & M_{\sigma}(v). 
}
\end{align}
Here the top arrow is an open immersion, the vertical 
arrows are natural morphisms to the coarse moduli spaces 
and the bottom arrow is induced by the universality of the 
coarse moduli spaces. 

Locally on $M_{\sigma}(v)$, we 
can compare the above diagram with a similar diagram for 
representations of the Ext-quiver with a convergent super-potential. 
For a closed point $p \in M_{\sigma}(v)$
corresponding to a polystable sheaf (\ref{poly}), 
let $Q_{E_{\bullet}}$ be the associated Ext-quiver. 
We take data (\ref{xi}) for the 
Ext-quiver $Q_{E_{\bullet}}$ by 
\begin{align*}
\xi=(\xi_i)_{1\le i\le k}, \ 
\xi_i= Z_{B^{+}, \omega^{+}}(E_i), \ 1\le i\le k. 
\end{align*}
Then we have the associated 
$\mu_{\xi}$-stability condition on
the category of 
$Q_{E_{\bullet}}$-representations. 
Let $\vec{m}$ be the dimension vector of $Q_{E_{\bullet}}$ given by
\begin{align}\label{dimvec:m}
m_i=\dim V_i, \ 1\le i\le k. 
\end{align}
We have the following result: 
\begin{thm}\label{thm:compare}\emph{(\cite[Theorem~7.7]{Todstack})}
For a closed point 
$p\in M_{\sigma}(v)$ corresponding to a polystable 
sheaf (\ref{poly}),
let $Q=Q_{E_{\bullet}}$ be the associated 
Ext-quiver. Then there exist
a convergent super-potential $W$ of $Q$ and 
 analytic 
open neighborhoods 
\begin{align*}
p \in T \subset M_{\sigma}(v), \ 
0 \in V \subset M_{Q}(\vec{m})
\end{align*}
where $\vec{m}$ is the dimension vector (\ref{dimvec:m}), 
such that 
the commutative diagram (\ref{dia:M+}) pulled back to $T$
\begin{align*}
\xymatrix{
r_M^{-1}(T) \ar@<-0.3ex>@{^{(}->}[r] 
 \ar[d]_{p_M^{+}}
 & p_M^{-1}(T)
\ar[d]^{p_M} \\
q_M^{-1}(T) \ar[r]_{q_M}  & T 
}
\end{align*}
is isomorphic to the commutative diagram
(see the diagram (\ref{dia:arrow})): 
\begin{align*}
\xymatrix{
 \mM_{(Q, \partial W)}^{\xi}(\vec{m})|_{V} 
\ar@<-0.3ex>@{^{(}->}[r]^{j_{(Q, \partial W)}^{\xi}}
\ar[d]_-{p_{(Q, \partial W)}^{\xi}} 
& \mM_{(Q, \partial W)}(\vec{m})|_{V} 
 \ar[d]^-{p_{(Q, \partial W)}} \\
M_{(Q, \partial W)}^{\xi}(\vec{m})|_{V} \ar[r]_-{q_{(Q, \partial W)}^{\xi}}
 & M_{(Q, \partial W)}(\vec{m})|_{V}.
}
\end{align*}
\end{thm}

By Theorem~\ref{thm:CYdcrit}, the stack 
$\mM_{\sigma}(v)$ has a canonical 
$d$-critical structure 
$s \in H^0(\sS_{\mM_{\sigma}(v)}^0)$, 
which induces 
the one on its open substack $p_M^{-1}(T)$.
The stack $\mM_{(Q, \partial W)}(\vec{m})|_{V}$
also has a $d$-critical structure by 
Proposition~\ref{prop:vir:tri}.
As we will mention in Remark~\ref{rmk:BPS0}, these $d$-critical 
structures should be the same. 
However we give a weaker statement for the
comparison of $d$-critical structures, which is 
enough for our purpose. 

\begin{prop}\label{prop:compare:d}
	Under the isomorphism 
\begin{align}\label{isom:tau}
\iota\colon \mM_{(Q, \partial W)}(\vec{m})|_{V}
\stackrel{\cong}{\to} p_M^{-1}(T)
\end{align}
in Theorem~\ref{thm:compare}, 
by shrinking $V$ if necessary, 
there is a $G$-invariant analytic function 
$\tr'(W) \colon \pi_Q^{-1}(V) \to \mathbb{C}$
and the identity in 
$H^0(\sS^0_{\mM_{(Q, \partial W)}(\vec{m})|_{V}})
=H^0(\sS^0_{\mathrm{Rep}_{(Q, \partial W)}(\vec{m})|_{V}})^G$ 
\begin{align}\label{id:dst}
\iota^{\ast}(s|_{p_M^{-1}(T)})=\tr' W+(d(\tr' W))^2. 
\end{align}
Here we have used the notation in (\ref{diagram:MQ}) for $Q=Q_{E_{\bullet}}$. 
\end{prop} 
\begin{proof}
	Note that $\iota^{\ast}(s|_{p_M^{-1}(T)})$ is a $G$-invariant 
	global section of $\oO_{\pi_Q^{-1}(V)}/I^2$
	where $I=(d\tr{W}) \subset \oO_{\pi_Q^{-1}(V)}$. 
	Since $\pi_Q^{-1}(V)$ is Stein and $G$ is reductive, 
	the map $H^0(\oO_{\pi_Q^{-1}(V)})^G \to H^0(\oO_{\pi_Q^{-1}(V)}/I^2)^G$
	is surjective. Therefore there is a $G$-invariant 
	analytic function $f \colon \pi_Q^{-1}(V) \to \mathbb{C}$
	such that $\iota^{\ast}(s|_{p_M^{-1}(T)})=f+I^2$. 
On the other hand, since we can take a minimum Darboux chart of a derived 
	enhancement of $\mM_{\sigma}(v)$ at $E$ (see~\cite[Theorem~2.10]{BBBJ}), 
	we can take a $d$-critical chart for 
	$\iota^{\ast}(s|_{p_M^{-1}(T)})$
	at $0 \in \pi_Q^{-1}(V)$
	of the form $(\{d\tr W=0\} \cap U, U, g, i\})$
	for an analytic open neighborhood $0 \in U \subset \pi_Q^{-1}(V)$
	and an analytic function $g \colon U \to \mathbb{C}$ which has 
	vanishing order at $0 \in \pi_{Q}^{-1}(V)$ at least three. 
	In particular on $U$, we have $I=(dg)$
	and 
	$\iota^{\ast}(s|_{p_M^{-1}(T)})=g+I^2$. 
	It follows that $f=g+I^2$ holds on $U$. 
	Then noting that $g$ has vanishing order at least three at $0$, 
	by shrinking $U$ if necessary, the above
	identity implies that $(df)=(dg)=I$ on $U$. 
	Since both of $f$, $\tr W$ are $G$-invariant, 
	it follows that $(df)=I$ holds 
	on $\cup_{\iota \in G}\iota^{-1}(U)$. 
	We take an analytic open neighborhood $0 \in V' \subset V$
	such that 
	$\pi_Q^{-1}(V') \subset \cup_{\iota \in G}\iota^{-1}(U)$ holds, 
	which is possible by~\cite[Lemma~5.1]{Todstack}. 
	By 
	setting $\tr'W=f|_{\pi_Q^{-1}(V')}$ and 
	replacing $V$ with $V'$, we obtain the proposition. 
	\end{proof}

\begin{rmk}\label{rmk:BPS0}
	It should be true that we can take 
	$\tr'W=\tr W$. 
	Indeed $(-1)$-shifted symplectic structure 
	for $\mM_{\sigma}(v)$
	is canonically determined by a left Calabi-Yau structure
	for a derived enhancement of $D^b \Coh(X)$,  
	which is a choice of a non-zero element of $H^0(X, K_X)=\mathbb{C}$
	(see~\cite{BrDy}). 
	It gives a $d$-critical structure in the left hand side of (\ref{id:dst}). 
	On the other hand, the left Calabi-Yau structure 
	 also determines cyclic $L_{\infty}$-structure on 
	a minimal model of $\RHom(E, E)[1]$, 
	which should give a standard Darboux form 
	for the formal completion of $\mM_{\sigma}(v)$ at $[E]$
	giving a $d$-critical structure 
	determined by $\tr W$
	(see~\cite[Appendix~A]{CosSch}). 
	By the canonicity of constructions of $(-1)$-shifted 
	symplectic structures, they should give the same $d$-critical 
	structures. 
	Some details may be pursued elsewhere. 	
	\end{rmk}

\begin{rmk}\label{rmk:BPS}
By the proof of Lemma~\ref{lem:cmor:phi}, 
under the isomorphism 
\begin{align*}
\overline{\tau} \colon M_{(Q, \partial W)}(\vec{m})|_{V} 
\stackrel{\cong}{\to} T
\end{align*}
in Theorem~\ref{thm:compare}, we have 
\begin{align*}
\overline{\tau}^{\ast}(\phi_{M_{\sigma}(v)|_{U}}|_{T})
=\phi_{\tr' \overline{W}}(j_{!\ast}\IC(V^s)).
\end{align*} 
Here $j \colon V^s \subset V$ is the 
simple part, 
$\phi_{M_{\sigma}(v)|_{U}}$ is given in Definition~\ref{def:def:phiM}, 
and we have used the notation of the diagram (\ref{com:MQ2}).
If $\tr'W=\tr W$, then the right hand side is 
nothing but the BPS sheaf defined in~\cite{DaMe}. 
 Therefore in this case, 
 the perverse sheaf $\phi_{M_{\sigma}(v)|_{U}}$
is interpreted as a gluing of BPS sheaves. 
\end{rmk}

\begin{rmk}\label{rmk:BPS2}
	We can replace $\tr W$ with $\tr'W$ in 
	Lemma~\ref{lem:phivanish}, Lemma~\ref{lem:cmor:phi} 
	so that we have the same results for vanishing
	cycle sheaves associated with $\tr'W$. 
	Namely let 
	\begin{align}\notag
		&\phi'_{\mM_{(Q, \partial W)}(\vec{m})|_{V}}
		\in \Perv(\mM_{(Q, \partial W)}(\vec{m})|_{V}), \\
		\notag&\phi'_{\mM_{(Q, \partial W)}^{\xi}(\vec{m})|_{V}}
		\in \Perv(\mM_{(Q, \partial W)}^{\xi}(\vec{m})|_{V})
	\end{align}
	be the perverse sheaves determined by the 
	d-critical structure (\ref{id:dst})
	and its restriction to $\mu_{\xi}$-semistable locus. 
	Then the same conclusions in Lemma~\ref{lem:phivanish}, Lemma~\ref{lem:cmor:phi} hold after replacing 
	$\phi_{\mM_{(Q, \partial W)}(\vec{m})|_{V}}, 
	\phi_{\mM_{(Q, \partial W)}^{\xi}(\vec{m})|_{V}}$
	with $\phi'_{\mM_{(Q, \partial W)}(\vec{m})|_{V}}, \phi'_{\mM_{(Q, \partial W)}^{\xi}(\vec{m})|_{V}}$
	respectively.  
	Indeed the only required property for $\tr W$ in 
	the above lemmas is a $G$-invariance so that it descends 
	to an analytic function on $V$. 
	As $\tr'W$ is also $G$-invariant, the same arguments apply. 
	\end{rmk}

\subsection{Independence of stability conditions}
We now prove the wall-crossing formula 
of the invariant $\Phi_{\sigma}(\gamma, m)$ 
defined in Definition~\ref{def:def:phiM}. 
Namely we show that it is independent of $\sigma$.  
\begin{thm}\label{thm:inde}
In the situation of Definition~\ref{def:def:phiM}, 
the Laurent polynomial 
$\Phi_{\sigma}(\gamma, m)$
 is independent of 
$\sigma \in U(X)$. 
\end{thm}
\begin{proof}
By the wall-chamber structure on $U(X)$, it is enough to show 
the following: 
for a fixed $\sigma \in U(X)$, if we take 
$\sigma^{+} \in U(X)$ to be sufficiently close to $\sigma$, 
then  
we have the identity
\begin{align}\label{id:GV}
\Phi_{\sigma}(\gamma, m)=\Phi_{\sigma^{+}}(\gamma, m).
\end{align}
Let $\gamma \in U \subset \Chow_X(\beta)$ be an open subset as in 
Definition~\ref{def:vir},
and consider stacks in the diagram (\ref{dia:chow2}).  
By pulling the diagram (\ref{dia:M+}) back to $U$, 
we obtain the commutative diagram
\begin{align}\notag
\xymatrix{
\mM_{\sigma^{+}}(v)|_{U} \ar@<-0.3ex>@{^{(}->}[r]^{j_M}
 \ar[dr]_{r_M}
 \ar[d]_{p_M^{+}}
 & \mM_{\sigma}(v)|_{U} 
\ar[d]^{p_M} \\
M_{\sigma^{+}}(v)|_{U} \ar[r]_{q_M}  & M_{\sigma}(v)|_{U}. 
}
\end{align}
We take a CY orientation data of $\mM_X(\beta)|_{U}$, 
which induces CY orientation data 
of its open substacks 
$\mM_{\sigma}(v)|_{U}, \mM_{\sigma^{+}}(v)|_{U}$. 
Let $\phi_{\mM_{\sigma}(v)|_{U}}$, 
$\phi_{\mM_{\sigma^{+}}(v)|_{U}}$
be the associated perverse sheaves (\ref{per:MU}) 
respectively. 
Then by our choice of orientation data, we have 
\begin{align*}
\phi_{\mM_{\sigma^{+}}(v)|_{U}}=
j_M^{\ast}\phi_{\mM_{\sigma}(v)|_{U}}. 
\end{align*}
Therefore we have the canonical morphism
\begin{align*}
\phi_{\mM_{\sigma}(v)|_{U}} \to
\dR j_{M\ast} j_M^{\ast}\phi_{\mM_{\sigma}(v)|_{U}} \to
 \dR j_{M\ast}\phi_{\mM_{\sigma^{+}}(v)|_{U}}.
\end{align*}
By pushing forward it to $M_{\sigma}(v)|_{U}$, we obtain the morphism
\begin{align}\label{mor:phiM}
\dR p_{M\ast}
\phi_{\mM_{\sigma}(v)|_{U}} \to 
\dR p_{M\ast}\dR j_{M\ast}\phi_{\mM_{\sigma^{+}}(v)|_{U}}
=\dR q_{M\ast} \dR p_{M\ast}^{+}\phi_{\mM_{\sigma^{+}}(v)|_{U}}. 
\end{align}
We claim that, for any $i \in \mathbb{Z}$
we have
\begin{align}\label{claim:M}
\dR q_{M\ast}\pH^i(\dR p_{M\ast}^{+}\phi_{\mM_{\sigma^{+}}(v)|_{U}})
\in \Perv(M_{\sigma}(v)|_{U}). 
\end{align}
Suppose that (\ref{claim:M}) holds. 
Then by taking the first perverse cohomologies of (\ref{mor:phiM}), 
we obtain the morphism
\begin{align*}
\pH^1(\dR p_{M\ast}
\phi_{\mM_{\sigma}(v)|_{U}})
\to 
\dR q_{M\ast} \pH^1(\dR p_{M\ast}^{+}\phi_{\mM_{\sigma^{+}}(v)|_{U}}).
\end{align*}
By Definition~\ref{def:def:phiM}, the above morphism is 
\begin{align}\label{phi:M:c}
\phi_{M_{\sigma}(v)|_{U}} \to 
\dR q_{M\ast}\phi_{M_{\sigma^{+}}(v)|_{U}}. 
\end{align}
We also claim that the morphism (\ref{phi:M:c}) 
is an isomorphism. 
In order to show the condition (\ref{claim:M}) and 
(\ref{phi:M:c}) is an isomorphism, it is enough 
to show these properties analytic locally on 
$M_{\sigma}(v)|_{U}$. 
By Theorem~\ref{thm:compare}, Proposition~\ref{prop:compare:d}
and Remark~\ref{rmk:BPS2}, 
we can essentially reduce these claims to claims 
for representations of 
symmetric quivers with convergent super-potentials. 
By Lemma~\ref{lem:phivanish} and 
Lemma~\ref{lem:cmor:phi}, and also noting Remark~\ref{rmk:BPS2}, 
we obtain the desired claims, i.e. 
(\ref{claim:M}) holds and (\ref{phi:M:c}) is an isomorphism. 

We have the commutative diagram
\begin{align*}
\xymatrix{
M_{\sigma^{+}}(v)|_{U} \ar[rr]^{q_M}
\ar[rd]_{\pi_M^{+}} & 
 & M_{\sigma}(v)|_{U} \ar[ld]^{\pi_M} \\
&  U  &
}
\end{align*}
where $\pi_M$ and $\pi_M^{+}$ are HC maps. 
By the isomorphism (\ref{phi:M:c}), we have the isomorphism
\begin{align*}
\dR \pi_{M\ast} \phi_{M_{\sigma}(v)|_{U}}
\stackrel{\cong}{\to}
\dR \pi_{M\ast}^{+} \phi_{M_{\sigma^{+}}(v)|_{U}}. 
\end{align*}
Therefore the
identity (\ref{id:GV}) holds
by the definition of $\Phi_{\sigma}(\gamma, m)$. 
\end{proof}

By Theorem~\ref{thm:inde}, we can 
define $\Phi_X(\gamma, m)$ by 
\begin{align*}
\Phi_X(\gamma, m) \cneq \Phi_{\sigma}(\gamma, m).
\end{align*}
for $\sigma \in U_X$. 
\subsection{Independence of Euler characteristics}
As an application of Theorem~\ref{thm:inde}, we show that 
$\Phi_X(\gamma, m)$ is also independent of $m$
when $\gamma$ is primitive. 
Here a one cycle $\gamma$ on $X$ is called \textit{primitive}
if it is written as 
\begin{align}\label{write:gamma}
\gamma=\sum_{i=1}^k a_i[C_i]
\end{align}
for irreducible curves $C_i$
with $C_i \neq C_j$ for $i\neq j$, 
 and $a_i \in \mathbb{Z}_{>0}$
such that 
the greatest common divisor of $(a_1, \ldots, a_k)$ is one. 

\begin{thm}\label{thm:indeE}
For a primitive one cycle $\gamma \in \Chow_X(\beta)$, suppose 
that $\mM_X(\beta)$ is CY at $\gamma$. 
Then $\Phi_X(\gamma, m)$ is independent of $m$. 
\end{thm}
\begin{proof}
We write $\gamma$ as (\ref{write:gamma}) and 
take divisors $D_i$ for $1\le i\le k$ defined 
in an analytic neighborhood of the support of $\gamma$
satisfying $D_i \cdot C_j=\delta_{ij}$. 
Since $\gamma$ is primitive, there is a divisor 
\begin{align*}
D=\sum_{i=1}^k d_i D_i,  \ d_i \in \mathbb{Z}
\end{align*}
such that $D \cdot \gamma=1$. 
Then for a sufficiently small 
analytic open neighborhood $\gamma \in U \subset \Chow_X(\beta)$, 
the map $F \mapsto F(D)$
gives an isomorphism
\begin{align}\label{isom:OD}
\otimes \oO(D) \colon 
\mM_{\sigma}(v)|_{U} \stackrel{\cong}{\to} \mM_{\sigma'}(v')|_{U}
\end{align}
which commutes with the HC maps to $U$. 
Here $\sigma=\sigma_{B, \omega}$, 
$\sigma'=\sigma_{B+D, \omega}$, 
$v=(\beta, m)$ and $v'=(\beta, m+1)$. 
The isomorphism (\ref{isom:OD}) is an isomorphism as 
$d$-critical 
stacks (see~\cite[Remark~8.6]{MT}). 
We take CY orientation data of $\mM_{\sigma}(v)|_{U}$
which induces the one on $\mM_{\sigma'}(v)|_{U'}$ 
by the isomorphism (\ref{isom:OD}). 
Then the isomorphism (\ref{isom:OD}) induces an isomorphism
\begin{align*}
M_{\sigma}(v)|_{U} \stackrel{\cong}{\to} M_{\sigma'}(v')|_{U}
\end{align*}
which commutes with HC maps to $U$, and 
sends $\phi_{M_{\sigma}(v)|_{U}}$ to 
$\phi_{M_{\sigma'}(v')|_{U}}$. 
Therefore we have 
$\Phi_{\sigma}(\gamma, m)=\Phi_{\sigma'}(\gamma, m+1)$. 
By Theorem~\ref{thm:inde}, 
we have $\Phi_X(\gamma, m)=\Phi_X(\gamma, m+1)$
and the theorem follows. 
\end{proof}

\section{Flop invariance of GV invariants}\label{sec:flop}
As an application of Theorem~\ref{thm:inde}, we show 
the flop invariance of generalized GV invariants. 
A similar result was already obtained for irreducible one cycles in~\cite{MT}, 
and the result here generalizes this result to arbitrary one cycles. 
\subsection{3-fold flops}
Let $X$, $X^{\dag}$ be smooth projective CY 3-folds. 
A diagram
\begin{align}\label{flop:dia}
\xymatrix{
X \ar[dr]_{f}  \ar@{-->}[rr]^{\phi} &  & X^{\dag} \ar[dl]^{f^{\dag}} \\
&  Y   &
}
\end{align}
is called a \textit{flop} if 
$f, f^{\dag}$ are projective 
birational morphisms 
which are isomorphic in codimension one
with relative Picard number one,  
$Y$ has only Gorenstein singularities, and 
$\phi$ is a non-isomorphic birational map.  
The exceptional loci of $f, f^{\dag}$ are 
chains of smooth rational curves.
In~\cite{Br1}, 
Bridgeland
showed that 
there is an equivalence of derived categories
(also see~\cite[Proposition~5.2]{ToBPS})
\begin{align}\label{D:equiv}
\Phi \colon D^b(\Coh_{\le 1}(X)) \stackrel{\sim}{\to}
 D^b(\Coh_{\le 1}(X^{\dag})). 
\end{align}
The above equivalence is 
given by the Fourier-Mukai transform 
whose kernel is  
$\oO_{X \times_Y X^{\dag}}$, 
and it fits into 
the commutative diagram
\begin{align*}
\xymatrix{
D^b(\Coh_{\le 1}(X)) \ar[r]^-{\Phi} \ar[d]_-{\ch} & D^b(\Coh_{\le 1}(X^{\dag})) \ar[d]^-{\ch} \\
\Gamma_X \ar[r]_{\Phi_{\Gamma}} & \Gamma_{X^{\dag}}.
}
\end{align*}
Here
$\Phi_{\Gamma}$ takes 
$(\beta, n)$ to $(\phi_{\ast}\beta, n)$, where 
$\phi_{\ast}\beta$ is characterized as 
$\phi_{\ast}\beta \cdot D=\beta \cdot \phi_{\ast}^{-1}D$
for any divisor $D$ on $X^{\dag}$ 
and $\phi_{\ast}^{-1}D$ is the strict 
transform. 
\begin{lem}
For any one cycle $\gamma$ on $X$, there is a 
one cycle $\phi_{\ast}\gamma$ on $X^{\dag}$ such that 
for any object $F \in D^b(\Coh_{\le 1}(X))$ with 
$l(F)=\gamma$, we have 
$l(\Phi(F))=\phi_{\ast}\gamma$.
\end{lem}
\begin{proof}
For $F, F' \in D^b(\Coh_{\le 1}(X))$, suppose that 
we have $l(F)=l(F')$. 
Then $[F]-[F']$ in 
$K(\Coh_{\le 1}(X))$ is 
represented by 
a linear combination of 
skyscraper sheaves. 
Therefore it is enough to show that 
$l(\Phi(\oO_x))=0$ for any $x \in X$. 
If $x \notin \Ex(f)$, then 
$\Phi(\oO_x)=\oO_{\phi(x)}$, so 
$l(\Phi(\oO_x))=0$ holds. 
If $x \in \Ex(f)$, then 
$l(\Phi(\oO_x))$ is a one cycle supported on 
$\Ex(f)$ whose homology class equals to zero 
in a neighborhood of $\Ex(f)$. 
Therefore $l(\Phi(\oO_x))=0$ and the lemma holds. 
\end{proof}

\subsection{Perverse coherent sheaves}
The equivalence (\ref{D:equiv}) 
restricts to the equivalence of perverse coherent sheaves, 
defined below.
\begin{defi}\emph{(\cite[Section~3]{Br1})}
The subcategories $\pPPer_{\le 1}(X/Y) \subset D^b(\Coh_{\le 1}(X))$
are defined by 
\begin{align}\label{p:heart}
\left\{ E \in D^b (\Coh_{\le 1}(X)) \colon 
\begin{array}{c}
\dR f_{\ast} E \in \Coh(Y) \\
\Hom^{<-p}(E, \cC_X)=\Hom^{<p}(\cC_X, E)=0 
\end{array} \right\}.
\end{align}
Here $\cC_X \cneq \{ F \in \Coh(X) : \dR f_{\ast} F=0\}$.
\end{defi}
The subcategories (\ref{p:heart}) are 
known to be the hearts of bounded t-structures
on $D^b(\Coh_{\le 1}(X))$. In particular they
are abelian categories. 
Indeed when $p \in \{-1, 0\}$, 
they are obtained as a tilting of $\Coh_{\le 1}(X)$, 
so any object in $\pPPer_{\le 1}(X/Y)$ is concentrated on 
$[-1, 0]$. 
Below we always take $p\in \{-1, 0\}$. 
These t-structures fit into Bridgeland stability conditions 
on the boundary points of $U(X)$ as follows: 
\begin{lem}\label{stabp}
Let us take data
\begin{align}\label{data:Hw}
\omega \in \mathrm{NS}(Y)_{\mathbb{R}}, \ H \in \mathrm{NS}(X)_{\mathbb{R}}
\end{align}
such that $\omega$ is ample and $H$ is $f$-ample. 
Then for $p\in \{-1, 0\}$,
there is $\delta_0>0$ such that 
for $0<\delta<\delta_0$ we have
\begin{align}\label{sigmap}
\ptau_{(-1)^{p+1} \delta H, \omega} 
\cneq (Z_{(-1)^{p+1}\delta H, f^{\ast}\omega}, 
\pPPer_{\le 1}(X/Y)) 
\in \overline{U(X)}. 
\end{align}
Here $Z_{-, -}$ is 
the group homomorphism 
$\Gamma_X \to \mathbb{C}$ 
defined as in (\ref{ZBw}). 
\end{lem}
\begin{proof}
See~\cite[Proposition~5.2]{ToBPS}. 
\end{proof}

The equivalence (\ref{D:equiv}) induces the 
isomorphism on the space of 
stability conditions
\begin{align}\label{isom:stab}
\Phi_{\ast} \colon \Stab_{\le 1}(X) \stackrel{\cong}{\to}
\Stab_{\le 1}(X^{\dag})
\end{align}
sending $(Z, \aA)$
to $(Z \circ \Phi^{-1}, \Phi(\aA))$. 
On the other hand, the equivalence (\ref{D:equiv}) 
restricts to the equivalence (see~\cite{Brs1})
\begin{align}\label{equiv:Per}
\Phi \colon \oPPer_{\le 1}(X/Y) \stackrel{\sim}{\to}
 \iPPer_{\le 1}(X^{\dag}/Y). 
\end{align}
Hence we have the following lemma: 
\begin{lem}\label{lem:Phisom}
Under the isomorphism (\ref{isom:stab}),
we have
\begin{align*}
\Phi_{\ast}\otau_{-\delta H, \omega}=
\itau_{\delta H^{\dag}, \omega}. 
\end{align*} 
Here $H^{\dag}=-\phi_{\ast}H$, which is 
$f^{\dag}$-ample by the definition of a flop. 
\end{lem}
\begin{proof}
See~\cite[Proposition~5.2]{ToBPS}
for details. 
\end{proof}

\subsection{Moduli stacks of semistable perverse coherent sheaves}
Let us write stability conditions (\ref{sigmap})
as 
\begin{align}\label{ptauZ}
\ptau=(\pZ, \pPPer_{\le 1}(X/Y)), \ p=-1, 0, 
\end{align}
 for simplicity. 
For $v=(\beta, m) \in \Gamma_X$, let 
$\mM_{\ptau}(v)$ be the moduli stack of 
$\ptau$-semistable objects
$E \in \pPPer_{\le 1}(X/Y)$ with 
$\ch(E)=v$. 
\begin{lem}\label{beta:m}
Suppose that $f_{\ast}\beta \neq 0$. Then we have 
$\mM_{\ptau}(v) \subset \mM_X(\beta)$, i.e. 
any object $[E] \in \mM_{\ptau}(v)$ is 
a pure one dimensional sheaf. In particular
if $\mM_{\ptau}(v) \neq \emptyset$, 
then 
$\beta$ is an effective class. 
\end{lem}
\begin{proof}
For $[E] \in \mM_{\ptau}(v)$, we have 
the exact sequence in $\pPPer_{\le 1}(X/Y)$
\begin{align*}
0 \to \hH^{-1}(E)[1] \to E \to \hH^0(E) \to 0. 
\end{align*}
Suppose that $\hH^{-1}(E) \neq 0$. 
Since $\hH^{-1}(E)$ is supported on 
$\Ex(f)$ and 
$\dR f_{\ast}E$ is a one dimensional sheaf on $Y$, we have
\begin{align*}
\pi=
\arg \pZ(\hH^{-1}(E)[1])>\arg \pZ(E). 
\end{align*}
The above inequality contradicts to the $\ptau$-semistability of $E$. 
So
we have $\hH^{-1}(E)=0$, and
$E \in \Coh_{\le 1}(X)$ holds. 
Similar argument shows that $E$ is a pure sheaf. 
\end{proof}
For $\sigma \in U(X)$ which is 
sufficiently close to $\ptau$, we have an open 
embedding
\begin{align}\label{open:m}
\mM_{\sigma}(v) \subset \mM_{\ptau}(v).
\end{align}
We show that the above 
embedding (\ref{open:m}) is an isomorphism 
for suitable choice of data (\ref{data:Hw}). 
We prepare using the following lemma: 
\begin{lem}\label{lem:prepare}
For an effective class $\beta \in N_1(X)$
with $f_{\ast}\beta\neq 0$, 
there exists $(H, \omega)$ in (\ref{data:Hw})
such that, 
after replacing $\delta_0>0$ by a smaller one if necessary
the following holds:
for any decomposition $\beta=\beta_1+\beta_2$
into effective classes $\beta_1$, $\beta_2$
which are not proportional in $N_1(X)_{\mathbb{R}}$,
$(m_1, m_2) \in \mathbb{Z}^2$
and $0<\pm \delta <\delta_0$,  
we have
\begin{align}\label{cond:delta}
\frac{m_1+\delta(H\cdot \beta_1)}{f^{\ast}\omega \cdot \beta_1}
\neq \frac{m_2+\delta(H\cdot \beta_2)}{f^{\ast}\omega \cdot \beta_2}.
\end{align}
 \end{lem}
\begin{proof}
We first note that
a general $(H, \omega)$
satisfies the following: 
for any decomposition $\beta=\beta_1+\beta_2$
into effective classes $\beta_1$, $\beta_2$
which are not proportional in $N_1(X)_{\mathbb{R}}$,
we have
\begin{align}\label{choice:Hw}
\alpha \cneq 
\frac{H \cdot \beta_1}{f^{\ast}\omega \cdot \beta_1} - 
\frac{H \cdot \beta_2}{f^{\ast}\omega \cdot \beta_2} \neq 0. 
\end{align}
Indeed, suppose that 
(\ref{choice:Hw}) fails for any $(H, \omega)$. 
Then for any $(D_1, D_2) \in \mathrm{NS}(X)_{\mathbb{R}}^{\times 2}$, 
we have 
\begin{align*}
(D_1 \cdot \beta_1) \cdot (D_2 \cdot \beta_2)=
(D_1 \cdot \beta_2) \cdot (D_2 \cdot \beta_1). 
\end{align*}
An easy linear algebra argument shows that the above 
condition implies that $\beta_1$ and $\beta_2$
are proportional. 
Therefore for a fixed non-proportional $(\beta_1, \beta_2)$, 
a general choice of $(H, \omega)$ satisfies (\ref{choice:Hw}). 
Since the possible pairs $(\beta_1, \beta_2)$ are finite, 
a general $(H, \omega)$ satisfies (\ref{choice:Hw})
for any decomposition $\beta=\beta_1+\beta_2$. 

Let us take $(H, \omega)$ as above such that 
$\omega$ is rational. 
If for $\delta \neq 0$ the condition (\ref{cond:delta}) fails, 
we have 
\begin{align*}
\delta=\frac{1}{\alpha}\left(\frac{m_2}{f^{\ast}\omega \cdot \beta_2}
-\frac{m_1}{f^{\ast}\omega \cdot \beta_1}
  \right). 
\end{align*}
Since the RHS takes discrete values, 
we can find $\delta_0>0$ such that 
for $0<\pm \delta<\delta_0$ the condition (\ref{cond:delta}) holds. 
\end{proof}

\begin{lem}\label{lem:Misom}
For an effective class $\beta \in N_1(X)$
with $f_{\ast}\beta\neq 0$, 
we take $(H, \omega)$ and $\delta_0>0$ as in Lemma~\ref{lem:prepare}. 
Then 
the open embedding (\ref{open:m}) is an isomorphism of stacks. 
\end{lem}
\begin{proof}
Let $E \in \pPPer_{\le 1}(X/Y)$ be a 
$\ptau$-semistable object with 
$\ch(E)=(\beta, m)$. 
By Lemma~\ref{beta:m}, we have 
$E \in \Coh_{\le 1}(X)$. 
Suppose that $E$ is not 
$\sigma$-semistable. 
Since $\Imm \pZ(E)>0$, there is an exact sequence
\begin{align}\label{E12}
0 \to E_1 \to E \to E_2 \to 0
\end{align}
in both of $\Coh_{\le 1}(X)$ and $\pPPer_{\le 1}(X/Y)$
which destabilizes $E$ in $\sigma$-stability
and
\begin{align*}
\arg \pZ(E_1)=\arg \pZ(E_2). 
\end{align*}
By setting $\ch(E_i)=(\beta_i, m_i)$, the above condition 
implies that $(\beta_1, m_1)$ and 
$(\beta_2, m_2)$ are proportional by Lemma~\ref{lem:prepare}. 
Then it contradicts to that (\ref{E12})
destabilizes $E$ in $\sigma$-stability. 
\end{proof}

\subsection{Flop invariance formula}
The following is the main result in this section. 
\begin{thm}\label{thm:flop}
Suppose that the stack $\mM_X(\beta)$ is CY 
at $\gamma \in \Chow_X(\beta)$ 
and $\Phi_X(\gamma, m)$ is non-zero. 
Then $\phi_{\ast}\gamma$ is an effective one cycle on $X^{\dag}$. 
If
$\mM_{X^{\dag}}(\phi_{\ast}\beta)$ is also CY 
at $\phi_{\ast}\gamma \in \Chow_{X^{\dag}}(\phi_{\ast}\beta)$,
then we have the identity
\begin{align}\label{id:flop}
\Phi_X(\gamma, m)=\Phi_{X^{\dag}}(\phi_{\ast}\gamma, m). 
\end{align}
\end{thm}
\begin{proof}
We take $(H, \omega)$
and $\delta_0>0$ as in Lemma~\ref{lem:Misom}. 
We write stability conditions
(\ref{ptauZ}) as $\ptau_X$, and 
let $\sigma_X^{(p)} \in U(X)$ be 
sufficiently close to $\ptau_X$. 
If $\Phi_X(\gamma, m)$ is non-zero, 
then by Theorem~\ref{thm:inde}
there is an object $[E] \in \mM_{\sigma_X^{(0)}}(\beta)$
such that $l(E)=\gamma$. 
By Lemma~\ref{lem:Phisom} and Lemma~\ref{lem:Misom}, 
we have the isomorphisms
\begin{align}\label{isom:MM0}
\mM_{\sigma_X^{(0)}}(\beta) \stackrel{\cong}{\hookrightarrow}
\mM_{\otau_X}(\beta) \stackrel{\Phi_{\ast}}{\to}
\mM_{\itau_X^{\dag}}(\phi_{\ast}\beta) \stackrel{\cong}{\hookleftarrow}
\mM_{\sigma_X^{\dag (-1)}}(\phi_{\ast}\beta). 
\end{align}
By the above isomorphisms, 
the object $\Phi(E)$ is also a sheaf so 
the one cycle $\phi_{\ast}\gamma=l(\Phi(E))$
on $X^{\dag}$ is effective. 
Let $T_X^{(p)}(\beta) \subset \Chow_X(\beta)$ be the image of 
the HC map 
$\mM_{\sigma_X^{(p)}}^{\rm{red}}(\beta) \to \Chow_X(\beta)$. 
We have the commutative diagram
\begin{align}\label{com:XX}
\xymatrix{
\mM_{\sigma_X^{(0)}}^{\rm{red}}(\beta)
 \ar[r]^-{\Phi_{\ast}}_-{\cong} \ar[d] & 
\mM_{\sigma_X^{\dag (-1)}}^{\rm{red}}(\phi_{\ast}\beta) \ar[d] \\
T_X^{(0)}(\beta) \ar[r]^-{\phi_{\ast}}_-{\cong}& T_{X^{\dag}}^{(-1)}(\phi_{\ast}\beta)
}
\end{align}
where the vertical arrows are HC maps. 
The isomorphisms (\ref{isom:MM0})
preserve the $d$-critical structures and 
the virtual canonical line bundles. 
Therefore 
$\mM_{\sigma_X^{\dag(-1)}}(\phi_{\ast}\beta)$
is 
also CY at  
$\phi_{\ast}\gamma$, and
 we have the identity
\begin{align*}
\Phi_{\sigma_X^{(0)}}(\gamma, m)=\Phi_{\sigma_{X^{\dag}}^{(-1)}}(\phi_{\ast}\gamma, m). 
\end{align*}
If we furthermore assume that $\mM_{X^{\dag}}(\phi_{\ast}\beta)$ is CY 
at $\phi_{\ast}\gamma$, then by
 Theorem~\ref{thm:inde} we obtain the identity (\ref{id:flop}). 
\end{proof}
By Theorem~\ref{thm:flop} and the 
identity (\ref{id:GV1}), we have the following corollary: 
\begin{cor}
For $\beta \in N_1(X)$
with $f_{\ast}\beta \neq 0$, suppose that 
Conjecture~\ref{conj:vir} holds for one cycles
$\gamma$ and $\phi_{\ast}\gamma$. 
Then the local GV invariant
$n_{g, \gamma}$
defined in~\cite{MT}
satisfy 
$n_{g, \gamma}=n_{g, \phi_{\ast}\gamma}$. 
In particular if Conjecture~\ref{conj:vir} holds for $X$ and $X^{\dag}$, 
then we have $n_{g, \beta}=n_{g, \phi_{\ast}\beta}$.  
\end{cor}
\section{The case of local surfaces}\label{sec:local}
The results so far are conditional to Conjecture~\ref{conj:vir}. 
In this section, we prove Conjecture~\ref{conj:vir}
for local surfaces, which shows similar results 
as Theorem~\ref{thm:inde}, Theorem~\ref{thm:indeE}
in this case
without assuming Conjecture~\ref{conj:vir}. 
\subsection{Moduli stacks of one dimensional semistable sheaves on local surfaces}
Let $S$ be a smooth projective surface 
and consider the non-compact CY 3-fold $X$
\begin{align*}
X\cneq \mathrm{Tot}_S(K_S)
\stackrel{p}{\to}S
\end{align*}
where $p$ is the projection. 
The above CY 3-fold is compactified by adding 
the section at the infinity: 
\begin{align*}
X \subset \overline{X} \cneq \mathbb{P}_S(\oO_S \oplus K_S). 
\end{align*}
Let $\Coh_{c, \le 1}(X)$ be the 
category of compactly supported coherent sheaves on 
$X$ whose supports have dimensions less than or equal to one. 
Note that $\Coh_{c, \le 1}(X)$ is the subcategory
\begin{align*}
\Coh_{c, \le 1}(X) \subset \Coh_{\le 1}(\overline{X})
\end{align*}
consisting of sheaves whose supports do not intersect with 
the divisor
$D_{\infty} \cneq \overline{X} \setminus X$. 

Let $s \colon S \to X$ be the zero section of $p$. 
For $\beta \in N_1(S)$, 
let $\mM_{\overline{X}}(s_{\ast}\beta)$ be 
the stack as in (\ref{moduli:2funct}) for $\overline{X}$, and 
\begin{align*}
\mM_{X}(\beta) \subset \mM_{\overline{X}}(s_{\ast}\beta)
\end{align*}
the open substack consisting of 
 sheaves whose supports do not intersect with 
$D_{\infty}$. 
By its construction, the 
stack $\mM_X(\beta)$ is
nothing but
the moduli stack of the following objects:
\begin{align*}
F \in \Coh_{c, \le 1}(X), \ [l(p_{\ast}F)]=\beta.
\end{align*}
By~\cite{PTVV} (also see~\cite[Theorem~5.2]{Bussi} for the noncompact CY 3-fold case),   
the stack $\mM_X(\beta)$ is a truncation of a smooth 
derived scheme with a $(-1)$-shifted symplectic structure. 
Therefore by~\cite{BBBJ}, there is a canonical 
$d$-critical structure on $\mM_X(\beta)$, with 
virtual canonical line bundle given by the same formula (\ref{vir:K})
as in the compact CY 3-fold case. 

Let $\Chow_X(\beta)$ be the open subscheme of $\Chow_{\overline{X}}(s_{\ast}\beta)$
consisting of one cycles $\gamma$ on $\overline{X}$
which do not intersect with $D_{\infty}$. 
The HC map
\begin{align}\label{HC:oX}
\mM_{\overline{X}}^{\rm{red}}(s_{\ast}\beta) \to \Chow_{\overline{X}}(s_{\ast}\beta)
\end{align}
for $\overline{X}$ restricts to the HC map
\begin{align}\label{piM:loc}
\pi_{\mM} \colon 
\mM_{X}^{\rm{red}}(\beta) \to \Chow_X(\beta)
\end{align}
by pulling back (\ref{HC:oX}) to 
the open locus 
$\Chow_X(\beta)$
in $\Chow_{\overline{X}}(s_{\ast}\beta)$. 

\subsection{CY property for local surfaces}
In this subsection, we show 
Conjecture~\ref{conj:vir} for 
the local surface case. 
\begin{thm}\label{thm:CYsurface}
For the local surface $X=\mathrm{Tot}_S(K_S)$, the stack
$\mM_X(\beta)$ is CY at any $\gamma \in \Chow_X(\beta)$. 
\end{thm}
\begin{proof}
Let us fix a one cycle $\gamma \in \Chow_X(\beta)$. 
We first note that there exist smooth 
curves $C_1, C_2$ on $S$ which
are not contained in the support of $p_{\ast}\gamma$,
 and 
 admit an isomorphism
\begin{align}\label{isom:omega}
\omega_S \cong \oO_S(C_2-C_1).
\end{align}
Indeed it is enough to take a sufficiently ample 
divisor $H$ on $S$ and take 
general sections 
$C_1 \in \lvert H \rvert$ and 
$C_2 \in \lvert H+K_S \rvert$. 
Below we fix such $C_1$, $C_2$ and an isomorphism
(\ref{isom:omega}). 

Let $T$ be a complex analytic space
and $\eE \in \Coh(X \times T)$
a 
$T$-flat family of coherent sheaves on $X$
giving a $T$-valued point 
of (analytification of) $\mM_X(\beta)$, 
i.e. a 1-morphism
$T \to \mM_X(\beta)$. 
Suppose that its composition with
the HC map (\ref{piM:loc}) is contained in a
sufficiently small open neighborhood of $\gamma
\in \Chow_X(\beta)$. 
We use the following commutative diagrams
\begin{align}\label{dia:locS}
\xymatrix{
X_T \cneq X \times T \ar[r]^-{\pi_T} \ar[d]_-{p_T} & T, \\
S_T \cneq S \times T \ar[ur]_-{q_T}
}
\quad
\xymatrix{
(C_i)_T \cneq C_i \times T \ar[r]^-{r_{T, i}} 
\ar@<-0.3ex>@{^{(}->}[d]_-{j_{T, i}}
& T \\
S_T = S \times T \ar[ur]_-{q_T}
}
\end{align}
where $\pi_T$, $p_T$, $q_T$, $r_{T, i}$ are the projections, 
and $j_{T, i}$ is the natural closed
embedding. 
 We have the canonical exact sequence of sheaves on 
$X_T$
\begin{align*}
0 \to p_T^{\ast}(\omega_S^{-1} \boxtimes p_{T\ast} \eE)
\to p_T^{\ast}p_{T\ast}\eE \to \eE \to 0. 
\end{align*}
We apply $\dR \hH om_{\pi_T}(-, \eE)$
to the above exact sequence. 
By setting
$\fF=p_{T\ast}\eE$ and using the adjunction, 
we obtain the 
distinguished triangle in $D^b(\Coh(T))$
\begin{align*}
\dR \hH om_{\pi_T}(\eE, \eE) \to 
\dR \hH om_{q_T}(\fF, \fF)
\to \dR \hH om_{q_T}(\omega_S^{-1} \boxtimes \fF, \fF). 
\end{align*}
Therefore we have the canonical isomorphism
\begin{align*}
K_{\mM_X(\beta)}^{\rm{vir}}|_{T} 
\stackrel{\cong}{\to}
\det \dR \hH om_{q_T}(\fF, \fF)
\otimes \det \dR \hH om_{q_T}(\omega_S^{-1}\boxtimes \fF, \fF)^{\vee}. 
\end{align*}
Using the isomorphism (\ref{isom:omega}), 
we have the distinguished triangle
\begin{align*}
	\oO_S(-C_1) \boxtimes \fF \to \omega_S \boxtimes \fF \to 
	\omega_S|_{C_2}\boxtimes \fF. 
	\end{align*}
By applying $\dR \hH om_{q_T}(\fF, -)$, we obtain 
the distinguished triangle 
\begin{align*}
	\dR \hH om_{q_T}(\oO_S(C_1)\boxtimes \fF, \fF)
	&\to \dR \hH om_{q_T}(\omega_S^{-1}\boxtimes \fF, \fF) \\
	&\to \dR \hH om_{r_{T, 2}}(\dL j_{T, 2}^{\ast}\fF, 
	\omega_S|_{C_2}\boxtimes \dL j_{T, 2}^{\ast}\fF)
		\end{align*}
which gives the isomorphism
\begin{align*}
&\det \dR \hH om_{q_T}(\omega_S^{-1}\boxtimes \fF, \fF)
\stackrel{\cong}{\to} \\
&\det \dR \hH om_{q_T}(\oO_S(C_1)\boxtimes \fF, \fF)
\otimes \det \dR \hH om_{r_{T, 2}}(\dL j_{T, 2}^{\ast}\fF, 
\omega_S|_{C_2}\boxtimes \dL j_{T, 2}^{\ast}\fF).
\end{align*}
Similarly from the distinguished triangle 
\begin{align*}
	\oO_S(-C_1) \boxtimes \fF \to \fF \to \oO_{C_1}\boxtimes \fF
	\end{align*}
we have the 
distinguished triangle
\begin{align*}
	\dR \hH om_{q_T}(\oO_S(C_1)\boxtimes \fF, \fF)
	&\to \dR \hH om_{q_T}(\fF, \fF) \\
	&\to \dR \hH om_{r_{T, 1}}(\dL j_{T, 1}^{\ast}\fF, 
	\dL j_{T, 1}^{\ast}\fF)
	\end{align*}
which gives the 
isomorphism
\begin{align*}
&\det \dR \hH om_{q_T}(\oO_S(C_1)\boxtimes \fF, \fF)  \\
&\stackrel{\cong}{\to} 
\det \dR \hH om_{q_T}(\fF, \fF)
\otimes \det \dR \hH om_{r_{T, 1}}(\dL j_{T, 1}^{\ast}\fF, 
\dL j_{T, 1}^{\ast}\fF)^{\vee}.
\end{align*}
By combining the above isomorphisms, 
we have the isomorphism
\begin{align*}
K_{\mM_X(\beta)}^{\rm{vir}}|_{T} 
\stackrel{\cong}{\to} &
\det \dR \hH om_{r_{T, 1}}(\dL j_{T, 1}^{\ast}\fF, 
 \dL j_{T, 1}^{\ast}\fF) \\
&\otimes
\det \dR \hH om_{r_{T, 2}}(\dL j_{T, 2}^{\ast}\fF, 
\omega_S|_{C_2}\boxtimes \dL j_{T, 2}^{\ast}\fF)^{\vee}.
\end{align*}
By our assumption that $C_i$ is not contained in the 
support of $p_{\ast}\gamma$, 
the object $\dL j_{T, i}^{\ast}\fF$
is a $T$-flat family of zero-dimensional sheaves on 
$C_i$. Moreover the above isomorphism is 
compatible with complex analytic maps $T' \to T$
for other complex analytic space $T'$. 
Therefore the result follows from Lemma~\ref{lem:curve} below. 
\end{proof}
We have used the following lemma: 
\begin{lem}\label{lem:curve}
Let $C$ be a smooth projective curve and 
$L$ a line bundle on it. 
Let $\mM_0$ be the stack of zero dimensional 
sheaves on $C$ 
and consider the HC map
\begin{align*}
\pi_{\mM_0} \colon 
\mM_0 \to \Sym(C)
\end{align*}
sending a zero-dimensional sheaf to its support. 
Let $\uU \in \Coh(C \times \mM_0)$ be 
the universal family, and 
$r_{\mM_0} \colon C \times \mM_0 \to \mM_0$ the projection. 
Let $K^{\rm{vir}}_{\mM_0, L}$ be the line bundle on 
$\mM_0$ defined by 

\begin{align*}
K^{\rm{vir}}_{\mM_0, L} \cneq 
\det \dR \hH om_{r_{\mM_0}}(\uU, L \boxtimes \uU).
\end{align*}
Then for each $[Z] \in \Sym(C)$, 
there is an analytic open neighborhood $[Z] \in U \subset \Sym(C)$
such that $K^{\rm{vir}}_{\mM_0, L}$ is trivial on 
$\pi_{\mM_0}^{-1}(U)$.  
\end{lem}
\begin{proof}
Let $[Z] \in \Sym(C)$ 
be given by $Z=\sum_{i=1}^k a_i[p_i]$ for 
distinct points
$p_1, \ldots, p_k \in C$
and $a_i \in \mathbb{Z}_{>0}$. 
We take analytic open neighborhoods
$p_i \in U_i \subset C$ such that 
$U_i \cap U_j =\emptyset$ for $i\neq j$
and $L|_{U_i}$ is trivial. 
Let $[Z] \in U \subset \Sym(C)$
be an open neighborhood given by 
\begin{align*}
U=\prod_{i=1}^k \Sym^{a_i}(U_i) \subset \Sym(C)
\end{align*}
where the right inclusion is given by 
$(Z_i)_{1\le i\le k} \mapsto \sum Z_i$. 
Then we have the isomorphism
\begin{align*}
\oplus \colon 
\prod_{i=1}^k \pi_{\mM_0}^{-1}(\Sym^{a_i}(U_i))
\stackrel{\cong}{\to} \pi_{\mM_0}^{-1}(U)
\end{align*}
given by taking the direct sum of zero dimensional
sheaves. Under the above isomorphism, we have
\begin{align*}
\oplus^{\ast}(K_{\mM_0, L}^{\rm{vir}}|_{\pi_{\mM_0}^{-1}(U)})
\cong \boxtimes_{i=1}^{k} 
K_{\mM_0, \oO_C}^{\rm{vir}}|_{\pi_{\mM_0}^{-1}(\Sym^{a_i}(U_i))}. 
\end{align*}

Therefore it is enough to show 
that $K_{\mM_0, \oO_C}^{\rm{vir}}$ is trivial 
when $C=\mathbb{A}^1$. 
In this case, the 
stack $\mM_0(k)$ of 
zero dimensional sheaves on $\mathbb{A}^1$
with length $k$ is given by 
\begin{align*}
\mM_0(k)=[\Hom(V, V)/\GL(V)]
\end{align*}
where $V$ is a $k$-dimensional vector space and 
$\GL(V)$ acts on $W \cneq \Hom(V, V)$ by conjugation. 
The universal sheaf $\uU$ is a $\GL(V)$-equivariant 
sheaf on $W \times \mathbb{A}^1$, which admits
a $\GL(V)$-equivariant exact sequence
\begin{align*}
0 \to V \otimes \oO_{W \times \mathbb{A}^1}
\stackrel{i}{\to} V \otimes \oO_{W \times \mathbb{A}^1} \to \uU \to 0. 
\end{align*}
Here the map $i$ corresponds to the 
$\GL(V)$-invariant section
of $\Hom(V, V) \otimes \oO_{W \times \mathbb{A}^1}$ given by
\begin{align*}
\delta \otimes 1-\id_V \otimes 1 \otimes t
\in \Hom(V, V) \otimes \Sym^{\bullet}(\Hom(V, V)^{\vee}) \otimes \mathbb{C}[t]
\end{align*}
where $\delta \in \Hom(V, V) \otimes \Hom(V, V)^{\vee}$ is the 
tautological element.  
Therefore $\uU$ is zero in the 
$\GL(V)$-equivariant K-theory of 
$\Hom(V, V)$, thus 
$K_{\mM_0, \oO_{C}}^{\rm{vir}}$ is trivial 
when $C=\mathbb{A}^1$. 

\end{proof}
\subsection{GV type invariants for local surfaces}
The GV type invariants for the local surface
$X=\mathrm{Tot}_S(K_S)$
is defined similarly to the projective CY 3-fold case. 
For an element
\begin{align*}
B+i\omega \in A(S)_{\mathbb{C}}
\end{align*}
and $F \in \Coh_{c, \le 1}(X)$, let 
$\mu_{B, \omega}(F) \in \mathbb{R} \cup \{\infty\}$ be defined by
\begin{align*}
\mu_{B, \omega}(F) \cneq \frac{\chi(F)-B \cdot l(p_{\ast}F)}
{\omega \cdot l(p_{\ast}F)}. 
\end{align*}
The above slope function defines the 
$(B, \omega)$-stability on 
$\Coh_{c, \le 1}(X)$.
Let $\Gamma_S \cneq N_1(S) \oplus \mathbb{Z}$
and
for $F \in \Coh_{c, \le 1}(X)$
we set
\begin{align}\label{ch:KS}
\ch(F) \cneq ([l(p_{\ast}F)], \chi(F)). 
\end{align}
Let $Z_{B, \omega}$ be the group homomorphism
$\Gamma_S \to \mathbb{C}$ defined by
\begin{align*}
Z_{B, \omega}(\beta, m) \cneq -m+(B+i\omega)\beta. 
\end{align*}
Then the pair 
\begin{align}\label{S:sigmaB}
\sigma_{B, \omega}  \cneq (Z_{B, \omega}, \Coh_{c, \le 1}(X))
\end{align}
is a Bridgeland stability condition on 
$D^b(\Coh_{c, \le 1}(X))$
w.r.t. the Chern character map (\ref{ch:KS}). 
Similarly to Subsection~\ref{subsec:twist2}, we denote 
by $\Stab_{\le 1}(X)$ the space of Bridgeland stability 
conditions on $D^b(\Coh_{c, \le 1}(X))$
w.r.t. the Chern character map (\ref{ch:KS}), and 
\begin{align*}
U(X) \subset \Stab_{\le 1}(X)
\end{align*}
the subset consisting of stability 
conditions of the form (\ref{S:sigmaB}). 

For $v=(\beta, m) \in N_1(S) \oplus \mathbb{Z}$, 
we have 
the open substack 
\begin{align*}
\mM_{\sigma}(v) \subset \mM_X(\beta)
\end{align*}
consisting of $\sigma$-semistable objects
 in $\Coh_{c, \le 1}(X)$. 
For $\gamma \in \Chow_X(\beta)$, let 
$\gamma \in U \subset \Chow_X(\beta)$ be a 
sufficiently small open neighborhood. 
Using the diagram (\ref{dia:chow2})
and CY orientation data of $\mM_X(\beta)|_{U}$
which exists by Theorem~\ref{thm:CYsurface}, 
the perverse sheaf
$\phi_{M_{\sigma}(v)|_{U}}$
on $\Perv(M_{\sigma}(v)|_{U})$
and the invariant
\begin{align}\label{loc:Phi}
\Phi_{\sigma}(\gamma, m) \in \mathbb{Z}[y^{\pm 1}]
\end{align}
are defined as in Definition~\ref{def:def:phiM}.
 Similarly to Lemma~\ref{lem:inde}, 
the invariant (\ref{loc:Phi})
is independent of a CY orientation data of $\mM_X(\beta)|_{U}$. 
Then the arguments of Theorem~\ref{thm:inde}
and Theorem~\ref{thm:indeE}
show the following (which is not conditional to Conjecture~\ref{conj:vir} by 
Theorem~\ref{thm:CYsurface}): 
\begin{thm}\label{thm:inde:loc}
Let $X=\mathrm{Tot}_S(K_S)$ for a smooth projective surface $S$. 
Then for $\sigma \in U(X)$, $\gamma \in \Chow_X(\beta)$ and 
$m \in \mathbb{Z}$,  
the invariant $\Phi_{\sigma}(\gamma, m)$
is independent of $\sigma$, 
so we can write it as $\Phi_X(\gamma, m)$. 
If furthermore $\gamma$ is primitive, 
then $\Phi_X(\gamma, m)$ is independent of $m$. 
\end{thm}

\subsection{Blow-up formula}
Let 
$S$ be a smooth projective surface and 
take a blow-up 
\begin{align*}
h \colon S^{\dag} \to S
\end{align*}
at a point $p\in S$. 
Then there exist smooth projective 3-folds $\overline{X}$, 
$\overline{X}^{\dag}$
connected by a flop 
\begin{align*}
\phi \colon \overline{X} \stackrel{f}{\to}
 Y \stackrel{f^{\dag}}{\leftarrow} \overline{X}^{\dag}
\end{align*}
satisfying the following conditions
(see~\cite[Lemma~4.2]{TodS}) 
\begin{itemize}
\item Both of the
 exceptional locus $Z=\Ex(f)$, $Z^{\dag} =\Ex(f^{\dag})$
are 
irreducible 
$(-1, -1)$-curves. 
\item There are closed embeddings 
\begin{align}\label{emb:S}
i \colon S \hookrightarrow \overline{X}, \quad 
i^{\dag} \colon S^{\dag} \hookrightarrow \overline{X}^{\dag}
\end{align}
such that 
$S \cap Z$ consists of one point, 
the strict transform of $S$
in $\overline{X}^{\dag}$ 
coincides with $S^{\dag}$, 
and $Z^{\dag} \subset S^{\dag}$ coincides with 
the exceptional locus of $h \colon S^{\dag} \to S$.   
\item There are open neighborhoods 
$S \subset X$, $S^{\dag} \subset X^{\dag}$
and isomorphisms
\begin{align}\label{isom:KS}
X \cong \mathrm{Tot}_S(K_S), \ 
X^{\dag} \cong \mathrm{Tot}_{S^{\dag}}(K_{S^{\dag}})
\end{align}
such that the embeddings (\ref{emb:S})
are identified with the zero sections. 
\end{itemize}
We regard 
one cycles on $S$, $S^{\dag}$
as one cycles on $X$, $X^{\dag}$
by isomorphisms (\ref{isom:KS}) and 
zero sections. 
Applying the 
argument of Theorem~\ref{thm:flop}, we obtain the 
following (which is not conditional to Conjecture~\ref{conj:vir}):  
\begin{thm}\label{thm:floploc}
Let $S$ be a smooth projective surface and 
$h \colon S^{\dag} \to S$ a blow-up at a point. 
Let $X=\mathrm{Tot}_S(K_S)$
and $X^{\dag}=\mathrm{Tot}_{S^{\dag}}(K_S^{\dag})$. 
Then for any effective one cycle $\gamma$ on $S$
and $m \in \mathbb{Z}$, we have the identity
\begin{align*}
\Phi_{X}(\gamma, m)=\Phi_{X^{\dag}}(h^{\ast}\gamma, m). 
\end{align*}
\end{thm}

\newcommand{\etalchar}[1]{$^{#1}$}
\providecommand{\bysame}{\leavevmode\hbox to3em{\hrulefill}\thinspace}
\providecommand{\MR}{\relax\ifhmode\unskip\space\fi MR }
% \MRhref is called by the amsart/book/proc definition of \MR.
\providecommand{\MRhref}[2]{%
	\href{http://www.ams.org/mathscinet-getitem?mr=#1}{#2}
}
\providecommand{\href}[2]{#2}

%\bibliographystyle{amsalpha}
%\bibliography{math}

Kavli Institute for the Physics and 
Mathematics of the Universe, University of Tokyo (WPI),
5-1-5 Kashiwanoha, Kashiwa, 277-8583, Japan.

\textit{E-mail address}: yukinobu.toda@ipmu.jp

\end{document}